\documentclass[final,onefignum,onetabnum]{siamonline190516}
\usepackage[english]{babel}
\usepackage[utf8]{inputenc}
\usepackage{graphics} 
\usepackage{blkarray, color}
\usepackage{algorithm}
\usepackage{amsfonts,amsmath,amssymb,bm}
\usepackage{preview}
\usepackage{hyperref}

 \newcommand \E {\mathbb{E}}
 \newcommand \PP {\mathcal{P}}
  \newcommand \rmP {\mathrm{P}}

 \newcommand \R {\mathbb{R}}
 \newcommand \NN {\mathcal{N}}
 \newcommand \MM {\mathcal{M}}
 \newcommand \N {\mathbb{N}}

\newcommand \ie {{\em i.e. }}

\newcommand \Id {\mathrm{I}_d}
\newcommand \1 {\mathbf{1}}
 \newcommand \KL {\mathrm{KL}}

 \usepackage{subfig}
\usepackage{mwe}

\newlength{\tempheight}
\newlength{\tempwidth}

\newcommand{\rowname}[1]
{\rotatebox{90}{\makebox[\tempheight][c]{\textbf{#1}}}}

\newcommand{\columnname}[1]
{\makebox[\tempwidth][c]{\textbf{#1}}}

\newtheorem{defi}{Definition}
\newtheorem{Corollary}{Corollary}
\newtheorem{prop}{Proposition}
\newtheorem*{propstar}{Proposition}
\newtheorem*{lemmastar}{Lemma}

\headers{A Wasserstein-type distance in the space of GMM}{J. Delon and A. Desolneux}
\title{A Wasserstein-type distance in the space of
  Gaussian Mixture Models\thanks{Submitted to the editors DATE.
\funding{This work was funded by the French National Research Agency
  under the grant ANR-14-CE27-0019 - MIRIAM.}}}

\author{Julie Delon\thanks{MAP5, Universit\'e de Paris, and Institut
    Universitaire de France (IUF)
  (\email{julie.delon@parisdescartes.fr})} \and Agnès
Desolneux\thanks{Centre Borelli, CNRS and ENS Paris-Saclay, France
  (\email{agnes.desolneux@math.cnrs.fr})} }

\ifpdf
\hypersetup{
pdftitle={A Wasserstein-type distance in the space of
  Gaussian Mixture Models}
  pdfauthor={J. Delon and A. Desolneux}
}
\fi

\begin{document}
\maketitle

\begin{abstract}
In this paper we introduce a Wasserstein-type distance on the set of
Gaussian mixture models. This distance is defined by restricting the
set of possible coupling measures in the optimal transport problem to
Gaussian mixture models. We derive a
very simple discrete formulation for this distance, which makes it suitable for
high dimensional problems. We also study the corresponding
multi-marginal and barycenter formulations.  We show some properties of
this Wasserstein-type distance, and we illustrate its practical use
with some examples in image processing.
\end{abstract}



\section{Introduction}

Nowadays, Gaussian Mixture Models (GMM) have become ubiquitous in
statistics and machine learning. These models are especially useful
in applied fields to represent probability distributions of real
datasets. Indeed, as linear combinations of Gaussian distributions,
they are perfect to model complex multimodal densities and can
approximate any continuous density when the numbers of components is 
chosen  large enough. Their parameters are also easy to infer with
algorithms such as the 
Expectation-Maximization (EM) algorithm~\cite{dempster1977maximum}. For instance, in
image processing, a large body of works use GMM to represent patch
distributions in images\footnote{Patches are small image pieces, they
  can be seen as vectors in a high dimensional space.}, and use these
distributions for various applications, such as image
restoration~\cite{ZoranWeiss2011,teodoro2015single,PLE,SPLE,houdard2018high,delon2018gaussian}
or texture synthesis~\cite{galerne2017semi}.   

The optimal transport theory provides mathematical tools to compare or
interpolate between probability distributions. For two probability
distributions $\mu_0$ and $\mu_1$ on $\R^d$ and a positive cost
function $c$ on $\R^d\times \R^d$, the goal is to solve the optimization problem 
\begin{equation}
  \label{eq:transport_definition}
\inf_{Y_0\sim \mu_0; Y_1\sim \mu_1} \E \left(c(Y_0,Y_1)\right),
\end{equation}
where the notation $Y\sim\mu$ means that $Y$ is a random variable with
probability distribution $\mu$. When $c(x,y) = \|x-y\|^p$ for $p\geq 1$,
Equation~\eqref{eq:transport_definition} (to a power $1/p$) defines
a distance between probability distributions that have a moment of order
$p$, called the Wasserstein distance $W_p$.

While this subject has gathered a lot of theoretical work
(see \cite{villanitopics, villani2008optimal, santambrogio2015optimal} for three 
reference monographies on the topic), its
success in applied fields was slowed down for many years by the
computational complexity of numerical algorithms which were not always
compatible with large amount of data. In recent years, the development
of efficient numerical approaches has been a game changer, widening
the use of optimal transport to various applications notably in image
processing, computer graphics and
machine learning~\cite{peyre2019computational}. However, computing Wasserstein distances or optimal
transport plans remains intractable when the dimension of the problem
is too high. 

Optimal transport can be used to compute distances or geodesics
between Gaussian mixture models, but optimal transport plans between
GMM, seen as probability distributions on a higher dimensional space, are
usually not Gaussian mixture models themselves, and the
corresponding Wasserstein geodesics between GMM do not preserve the
property of being a GMM. In order to keep the good properties of these models, 
we define in this paper a variant of the Wasserstein distance by
restricting the set of possible coupling measures to Gaussian mixture
models. The idea of restricting the set of possible coupling measures
has already been explored for instance
in~\cite{bionnadal2019}, where the distance is defined on
the set of the probability distributions of strong solutions to
stochastic differential equations. The goal of the authors is to
define a distance which keeps the good properties of $W_2$ while being
numerically tractable.  

In this paper, we  show that restricting the set of possible coupling
measures to Gaussian mixture models transforms the original infinitely
dimensional optimization problem into a finite dimensional problem
with a simple discrete formulation, depending only on the parameters
of the different Gaussian distributions in the mixture. When the
ground cost is simply $c(x,y) = \|x-y\|^2$, this yields a geodesic
distance, that we call $MW_2$ (for Mixture Wasserstein), which is obviously larger than $W_2$,
and is always upper bounded by $W_2$ plus a term depending only on the
trace of the covariance matrices of the Gaussian components in the mixture. The
complexity of the corresponding discrete optimization problem does not
depend on the space dimension, but only on the number of components in
the different mixtures, which makes it particularly suitable in
practice for high dimensional problems. Observe that this equivalent discrete
formulation has been proposed twice recently in the machine
learning literature, but with a very different point of view, by two independent
teams~\cite{chen2016distance,chen2019aggregated} and
\cite{chen2017optimal,chen2019optimal}.

Our original contributions in this paper are the following:
\begin{enumerate}
\item We derive an explicit formula for the optimal
  transport between two GMM restricted to GMM couplings, and we show
  several properties of the resulting distance, in particular how it
  compares to the classical Wasserstein distance.
\item We study the multi-marginal 
and barycenter formulations of the problem, and show the link between
these formulations.
\item We propose a generalized formulation to be used on distributions
  that are not GMM.
\item We provide two applications in image processing, respectively to
  color transfer and texture synthesis.
\end{enumerate}

The paper is organized as follows. Section~\ref{sec:background} is a reminder on Wasserstein distances and barycenters between probability measures on $\R^d$. We also recall the explicit formulation of $W_2$ between Gaussian distributions. In Section~\ref{sec:GMM_prop}, we recall some properties of Gaussian mixture models, focusing on an identifiabiliy property that will be necessary for the rest of the paper. We also show that optimal transport plans for $W_2$ between GMM are generally not GMM themselves. Then, Section~\ref{sec:MW2} introduces the $MW_2$ distance and derives the corresponding discrete formulation.  Section~\ref{sec:comparison} compares $MW_2$ with $W_2$. 
Section~\ref{sec:multimarginal} focuses on the corresponding
multi-marginal and barycenter formulations.   In
Section~\ref{sec:assignment}, we explain how to use $MW_2$ in practice
on distributions that are not necessarily GMM.
We conclude in
Section~\ref{sec:applications} with two applications of the distance $MW_2$ to image processing.     
To help the reproducibility of the results we present in this paper,
we have made our Python codes available on the Github website \url{https://github.com/judelo/gmmot}.

\section*{Notations}

We define in the following some of the notations that will be used in the paper.
\begin{itemize}
  \item The notation $Y\sim\mu$ means that $Y$ is a random variable with
probability distribution $\mu$.
\item If $\mu$ is a positive measure on a space $\mathcal{X}$ and $T:\mathcal{X}\to\mathcal{Y}$ is an
  application, $T\#\mu$ stands for the push-forward measure of $\mu$
  by $T$, {\em i.e.} the measure on $\mathcal{Y}$ such that
  $\forall A \subset \mathcal{Y} $, $(T\#\mu)(A) = \mu(T^{-1}(A))$.

  \item The notation $\mathrm{tr}(M)$ denotes the trace of the matrix $M$.
 \item The notation $\mathrm{Id}$ is the identity application.
\item $\langle\xi,\xi'\rangle $ denotes the Euclidean scalar product between $\xi$ and $\xi'$ in $\R^d$
\item $\MM_{n,m}(\R)$ is the set of real matrices with $n$ lines and
  $m$ columns, and we denote by $ \MM_{n_0,n_1,\dots,n_{J-1}}(\R)$ the set of
  $J$ dimensional tensors of size $n_k$ in dimension $k$.
\item $\1_n = (1,1,\dots,1)^t$ denotes a column vector of ones of length $n$.
\item For a given vector $m$ in $\R^d$ and a $d\times d$ covariance
  matrix $\Sigma$, $g_{m,\Sigma}$ denotes the density of the Gaussian
  (multivariate normal) distribution $\mathcal{N}(m,\Sigma)$.  
\item When $a_i$ is a finite sequence of $K$ elements (real numbers, vectors
  or matrices), we denote its elements as $a_i^0,\ldots, a_i^{K-1}$. 
\end{itemize}

\section{Reminders: Wasserstein distances and barycenters between probability measures on $\R^d$}
\label{sec:background}
Let $d\geq 1$ be an integer. We recall in this section the definition
and some basic properties of the Wasserstein distances between probability measures on $\R^d$. 
We write $\mathcal{P}(\R^d)$ the set probability measures on $\R^d$. For $p\geq 1$, the Wasserstein space $\mathcal{P}_p(\R^d)$ is defined as the set of probability measures $\mu$ with a finite moment of order $p$, {\em i.e.} such that  
\[\int_{\R^d} \|x\|^pd\mu(x) < +\infty,\]
with $\|.\|$ the Euclidean norm on $\R^d$.\\
For $t\in[0,1]$, we define $\mathrm{P}_t:\R^d\times \R^d \rightarrow \R^d$ by 
$$\forall x,y \in\R^d , \quad \rmP_t(x,y) = (1-t) x + t y \in\R^d .$$
Observe that $\mathrm{P}_0$ and $\mathrm{P}_1$ are the projections from  $\R^d\times \R^d$ onto $\R^d$ such that $\mathrm{P}_0(x,y) = x$ and $\mathrm{P}_1(x,y) = y$.  

\subsection{Wasserstein distances}

Let $p\geq 1$, and let $\mu_0, \mu_1$ be two probability measures in $\mathcal{P}_p(\R^d)$.  
 Define $\Pi(\mu_0,\mu_1)\subset \mathcal{P}_p(\R^d\times \R^d)$ as
 being the subset of probability distributions $\gamma$ on $\R^d\times \R^d$ with marginal distributions $\mu_0$ and $\mu_1$, \ie such that $\mathrm{P}_0\# \gamma = \mu_0$ and  $\mathrm{P}_1\# \gamma = \mu_1$. 
The \textit{$p$-Wasserstein distance} $W_p$ between $\mu_0$ and $\mu_1$ is defined as 
\begin{equation}
  \label{eq:wasserstein_definition}
   W_p^p(\mu_0,\mu_1) :=  \inf_{Y_0\sim \mu_0; Y_1\sim \mu_1} \E \left(\|Y_0 - Y_1\|^p\right) =  \inf_{\gamma \in \Pi(\mu_0,\mu_1)} \int_{\R^d\times\R^d} \|y_0-y_1\|^pd\gamma(y_0,y_1).
 \end{equation}
 This formulation is a special case of~\eqref{eq:transport_definition}
 when $c(x,y) = \|x-y\|^p$. 
 It can be shown (see for instance~\cite{villani2008optimal}) that
 there is always a couple $(Y_0,Y_1)$ of random variables which
 attains the infimum (hence a minimum) in the previous energy. Such a
 couple is called an \textit{optimal coupling}. The probability distribution $\gamma$ of
 this couple is called an \textit{optimal transport plan} between $\mu_0$ and $\mu_1$. This plan distributes all the mass of the distribution $\mu_0$ onto the distribution $\mu_1$ with a minimal cost, and the quantity $W_p^p(\mu_0,\mu_1)$ is the corresponding total cost. 

As suggested by its name ($p$-Wasserstein distance), $W_p$ defines a
metric on $\mathcal{P}_p(\R^d)$. It also metrizes the weak
convergence\footnote{A sequence $(\mu_k)_k$ converges weakly to $\mu$ in
  $\mathcal{P}_p(\R^d)$ if  it converges to $\mu$ in the sense of
  distributions and if $\int \|y\|^pd\mu_k(y)$ converges to $\int
  \|y\|^pd\mu(y)$.} in $\mathcal{P}_p(\R^d)$
(see~\cite{villani2008optimal}, chapter 6). It follows that $W_p$ is
continuous on $\PP_p(\R^d)$ for the topology of weak convergence. 

From now on, we will mainly focus on the case $p=2$,
since $W_2$ has an explicit formulation if $\mu_0$ and $\mu_1$ are
Gaussian measures.

\subsection{Transport map, transport plan and displacement interpolation}  
Assume that $p=2$. When $\mu_0$ and $\mu_1$ are two probability distributions on $\R^d$ and assuming that $\mu_0$ is absolutely continuous, then it can be shown that the optimal transport plan $\gamma$ for the problem~\eqref{eq:wasserstein_definition} is unique and has the form 
\begin{equation}
\gamma = (\mathrm{Id},T)\# \mu_0,\label{eq:transport_map}
\end{equation}
where $T:\R^d \mapsto \R^d$ is an application called {\it optimal transport map} and satisfying $T\# \mu_0=\mu_1$ (see~\cite{villani2008optimal}).

If $\gamma$ is an optimal transport plan for $W_2$ between two probability distributions $\mu_0$ and
$\mu_1$, the path $(\mu_t)_{t\in[0,1]}$ given by 
$$\forall t\in[0,1], \quad \mu_t := \rmP_t\# \gamma $$
defines a constant speed geodesic in
$\mathcal{P}_2(\R^d)$ (see for instance~\cite{santambrogio2015optimal} Ch.5, Section 5.4).
The path $(\mu_t)_{t\in[0,1]}$ is called the displacement interpolation
between $\mu_0$ and $\mu_1$ and it satisifes 
\begin{equation}
  \label{eq:ot_geodesic}
\mu_t \in \, \, \mathrm{argmin}_\rho \,\,  (1-t) W_2(\mu_0,\rho)^2 + t
W_2(\mu_1,\rho)^2 .
\end{equation}

This interpolation, often called Wasserstein barycenter in the
literature, can be easily extended to more than two probability
distributions, as recalled in the next paragraphs.

\subsection{Multi-marginal formulation and barycenters}

For $J\geq2$, for a set of weights $\lambda = (\lambda_0,\dots,\lambda_{J-1})\in
(\R_+)^J$ such that $\lambda \1_{J}= \lambda_0+\ldots + \lambda_{J-1}
=1$ and for $x = (x_0,\dots,x_{J-1})\in (\R^d)^J$,  we write  
\begin{equation}
B(x) = \sum_{i=0}^{J-1}\lambda_i x_i = \mathrm{argmin}_{y\in\R^d}
\sum_{i=0}^{J-1}\lambda_i \| x_i -y\|^2
\label{eq:barycenter}
\end{equation}
the barycenter of the $x_i$ with  weights $\lambda_i$. 

For $J$ probability distributions $\mu_0,\mu_1\dots,\mu_{J-1}$ on $\R^d$, we say that $\nu^*$ is the barycenter of the $\mu_j$ with weights $\lambda_j$ if $\nu^*$ is solution of 
  \begin{equation}
    \label{eq:W2_bary}
    \inf_{\nu \in \PP_2(\R^d)} \sum_{j=0}^{J-1} \lambda_j W_2^2(\mu_j,\nu).
  \end{equation}

Existence and unicity of barycenters for $W_2$ has been studied in
depth by Agueh and Carlier in~\cite{agueh2011barycenters}. They show
in particular that if one of the $\mu_j$ has a density, this
barycenter is unique. They also show that the solutions of the
barycenter problem are related to the solutions of the
multi-marginal transport problem (studied by Gangbo and \'Swi\'ech in~\cite{gangbo1998optimal})   
{\small \begin{eqnarray}
  mmW_2^2(\mu_0,\dots,\mu_{J-1})
& = & \inf_{\gamma\in\Pi(\mu_0,\mu_1,\dots,\mu_{J-1})} \int_{\R^d\times \dots \times \R^d}\frac{1}{2}
      \sum_{i,j=0}^{J-1}\lambda_i\lambda_j \| y_i-y_j\|^2
      d\gamma(y_0,y_1,\ldots,y_{J-1}) , \label{eq:W2_multimarge_definition}
\end{eqnarray}}
where $\Pi(\mu_0,\mu_1,\dots,\mu_{J-1})$ is the set of probability
measures on $(\R^d)^{J}$ having $\mu_0,\mu_1,\dots,\mu_{J-1}$ as
marginals.  More precisely, they show that
if~\eqref{eq:W2_multimarge_definition} has a solution $\gamma^*$,
then $\nu^* = B\# \gamma^*$ is a solution of~\eqref{eq:W2_bary}, and
the infimum of~\eqref{eq:W2_multimarge_definition} 
and~\eqref{eq:W2_bary} are equal.

\subsection{Optimal transport between Gaussian distributions}

Computing optimal transport plans between probability distributions is
usually difficult. In some specific cases, an  explicit solution is
known. For instance, in the one dimensional ($d=1$) case, when the cost $c$ is a convex function of the Euclidean distance on the line, the optimal
plan consists in a monotone rearrangement of the distribution $\mu_0$
into the distribution $\mu_1$ (the mass is transported monotonically
from left to right, see for instance Ch.2, Section 2.2 of
\cite{villanitopics} for all the details). Another case where the
solution is known for a quadratic cost is the 
Gaussian case in any dimension $d\geq 1$.

\subsubsection{Distance $W_2$ between Gaussian distributions}
If  $\mu_i= \mathcal{N}(m_i,\Sigma_i)$, $i\in\{0,1\}$ are two Gaussian distributions on $\R^d$, the $2$-Wasserstein distance $W_2$ between $\mu_0$ and $\mu_1$ has a closed-form expression, which can be written 
\begin{equation}
  \label{eq:wasserstein_gaussian}
  W_2^2(\mu_0,\mu_1) = \|m_0-m_1\|^2 + \mathrm{tr}\left( \Sigma_0 + \Sigma_1  - 2 \left(\Sigma_0^{\frac 1 2}\Sigma_1\Sigma_0^{\frac 1 2}\right)^{{\frac 1 2}} \right),
\end{equation}
where, for every symmetric semi-definite positive matrix $M$, the
matrix $M^{\frac 1 2}$ is its unique semi-definite positive square root.

If $\Sigma_0$ is non-singular, then the optimal map $T$ between $\mu_0$ and $\mu_1$ turns out to be affine and is  given by 
\begin{equation}
 \forall x\in\R^d , \quad T(x) = m_1+\Sigma_0^{-\frac 1
   2}\left(\Sigma_0^{\frac 1 2}\Sigma_1\Sigma_0^{\frac 1
     2}\right)^{{\frac 1 2}}\Sigma_0^{-\frac 1 2}(x-m_0) = m_1 +
 \Sigma_0^{-1}(\Sigma_0\Sigma_1)^{\frac 1 2} (x-m_0) ,
\label{eqT:eq}
\end{equation}
and the optimal plan $\gamma$ is then a Gaussian distribution on
$\R^d\times\R^d=\R^{2d}$ that is degenerate since it is supported by
the affine line $y=T(x)$.
These results have been known since~\cite{dowson1982frechet}.

Moreover, if $\Sigma_0$ and $\Sigma_1$ are non-degenerate, the
geodesic path $(\mu_t)$, $t\in(0,1)$, 
between $\mu_0$ and $\mu_1$ is given by
$\mu_t = \NN(m_t,\Sigma_t)$ with 
$m_t = (1-t)m_0 +tm_1$ and 
\[\Sigma_t = ((1-t) \Id +tC)\Sigma_0 ((1-t) \Id+tC),\]
with $\Id$ the $d\times d$ identity matrix and $C = \Sigma_1^{\frac 1 2}\left(\Sigma_1^{\frac 1
    2}\Sigma_0\Sigma_1^{\frac 1 2}\right)^{{-\frac 1 2}}\Sigma_1^{\frac
  1 2}$. 

This property still holds if the covariance matrices are not invertible, by replacing the inverse by the Moore-Penrose pseudo-inverse matrix, see Proposition 6.1 in~\cite{xia2014synthesizing}. The optimal map $T$ is not generalized in this case since the optimal plan is usually not supported by the graph of a function.  

\subsubsection{$W_2$-Barycenters in the Gaussian case}
\label{sec:multimarginal_gaussian}
For $J\geq 2$, let $\lambda = (\lambda_0,\dots,\lambda_{J-1})\in
(\R_+)^J$ be a set of positive weights summing to $1$ and let
$\mu_0,\mu_1\dots,\mu_{J-1}$ be $J$  Gaussian probability
distributions on $\R^d$. For $j=$ $0 \dots J-1$, we denote by $m_j$
and $\Sigma_j$ the expectation and the covariance matrix of $\mu_j$.
Theorem 2.2 in~\cite{ruschendorf2002n} tells us that if the covariances
$\Sigma_{j}$ are all positive definite, then the solution of the multi-marginal problem~\eqref{eq:W2_multimarge_definition} for the Gaussian distributions $\mu_0,\mu_1\dots,\mu_{J-1}$ can be written 
\begin{equation}
  \label{eq:gamma_optimal_multimarge}
\gamma^*(x_0,\dots,x_{J-1}) =
g_{m_0,\Sigma_0}(x_0)\, \delta_{(x_1,\dots,x_{J-1})=(S_1S_0^{-1}x_0,\dots,S_{J-1}S_0^{-1}x_0)}
\end{equation}
where
$S_j =  \Sigma_{j}^{1/2}\left(\Sigma_{j}^{1/2}\Sigma_*\Sigma_{j}^{1/2}\right)^{-1/2}\Sigma_{j}^{1/2}$
with $\Sigma_*$ a
solution of the fixed-point problem
\begin{equation}
  \label{eq:fixed}
  \sum_{j=0}^{J-1}  \lambda_j\left(\Sigma_*^{1/2}
    \Sigma_{j}\Sigma_*^{1/2}\right)^{1/2} = \Sigma_*.
\end{equation}
The barycenter $\nu^*$ of all the $\mu_j$ with weights $\lambda_j$ is the distribution $\NN(m_*,\Sigma_*)$, with $m_* = \sum_{j=0}^{J-1} \lambda_j m_j$.   
Equation~\eqref{eq:fixed} provides a natural iterative algorithm
(see \cite{alvarez2016fixed}) to compute the fixed point $\Sigma_*$ from the set of covariances $\Sigma_j$, $j\in \{0,\dots,J-1\}$.

\section{Some properties of Gaussian Mixtures Models}
\label{sec:GMM_prop}

The goal of this paper is to investigate how the optimisation
problem~\eqref{eq:wasserstein_definition} is transformed when the
probability distributions $\mu_0$, $\mu_1$ are finite Gaussian mixture
models and the transport plan $\gamma$ is forced to be a Gaussian
mixture model. This will be the aim of Section~\ref{sec:MW2}. Before,
we first need to recall a few basic properties on these mixture
models, and especially a density property and an identifiability
property. 

In the following, for $N\geq 1$ integer, we define the simplex
$$ \Gamma_{N} = \{\pi\in \R^{N}_+\;;\;\pi\1_N  = \sum_{k=1}^{N} \pi_k =
1\} .$$

\begin{defi}
  Let $K\geq 1$ be an integer. A (finite) Gaussian mixture model of size $K$ on $\R^d$ is a probability distribution $\mu$ on $\R^d$ that can be written
\begin{equation}
  \label{eq:GMM}
  \mu =\sum_{k=1}^{K} \pi_k \mu_k \;\text{ where }\; \mu_k = \mathcal N (m_k, \Sigma_k)  \text{ and } \pi \in \Gamma_{K}.
\end{equation}
 \end{defi}
We write $GMM_d(K)$ the subset of $\mathcal{P}(\R^d)$ made of
probability measures on $\R^d$ which can be written as Gaussian
mixtures with less than $K$ components (such mixtures are obviously
also in $\mathcal{P}_p(\R^d)$ for any $p\geq 1$). For $K < K'$, $GMM_d(K)\subset
GMM_d(K').$ The set of all finite Gaussian mixture distributions is written 
$$GMM_d(\infty) = \cup_{K\geq 0} GMM_d(K) .$$

\subsection{Two properties of GMM}
 
The following lemma states that any  measure in $\mathcal{P}_p(\R^d)$
can be approximated with any precision for the distance $W_p$ by a
finite convex combination of Dirac masses. This classical result will be useful
in the rest of the paper. 

\begin{lemma}
\label{prop:Dirac_dense_Wp}
The set 
\[ \left\{ \sum_{k=1}^N \pi_k \delta_{y_k}\;;\; N\in\N,\;(y_k)_k \in (\R^{d})^N, \; (\pi_k)_k\in \Gamma_N \right\}\]
is dense in $\mathcal{P}_p(\R^d)$ for the metric $W_p$, for any $p\geq 1$. 
\end{lemma}
 For the sake of completeness, we provide a
  proof in Appendix, adapted from the proof of Theorem 6.18 in~\cite{villani2008optimal}.
Since Dirac masses can be seen as degenerate Gaussian distributions, a direct consequence of Lemma~\ref{prop:Dirac_dense_Wp} is the following proposition.  
\begin{prop}
\label{prop:density_gmm}
 $GMM_d(\infty)$ is dense in $\mathcal{P}_p(\R^d)$ for the metric $W_p$.
\end{prop}

Another important property will be necessary, related to the
identifiability of  Gaussian mixture models.
It is clear such  models are not \textit{stricto
  sensu} identifiable, since  reordering the indexes of a mixture
changes its parametrization without changing the underlying
probability distribution, or also because a component with mass $1$ can be
divided in two identical components with masses $\frac 1 2$, for
example. However, if we write mixtures in a
``compact'' way (forbidding two components of the same mixture to be
identical), identifiability holds, up to a reordering of the
indexes. This property is 
reminded below.

\begin{prop}
The set of finite Gaussian mixtures is identifiable, in the sense that two mixtures $\mu_0 =\sum_{k=1}^{K_0} \pi_0^k \mu_0^k$ and $\mu_1= \sum_{k=1}^{K_1} \pi_1^k \mu_1^k$, written such that all $\{\mu_0^k\}_k$ (resp. all $\{\mu_1^j\}_j$) are pairwise distinct, are equal if and only if $K_0=K_1$ and we can reorder the indexes such that for all $k$, $\pi_0^k = \pi_1^k$, $m_0^k = m_1^k$ and $\Sigma_0^k = \Sigma_1^k$. 
  \label{lemma:unicity_gaussian}
\end{prop}
This result is also classical and the proof is provided in Appendix.

\subsection{Optimal transport and Wasserstein barycenters between Gaussian Mixture Models}

We are now in a position to investigate optimal transport between
Gaussian mixture models (GMM). A first important remark is that given two Gaussian mixtures $\mu_0$ and $\mu_1$ on $\R^d$, optimal transport plans $\gamma$ between  $\mu_0$ and $\mu_1$ are usually not GMM.

\begin{prop}
Let $\mu_0 \in GMM_d(K_0)$ and $\mu_1 \in GMM_d(K_1)$ be two Gaussian mixtures such that $\mu_1$ cannot be written $T\#\mu_0$ with $T$ affine. Assume also that $\mu_0$ is absolutely continuous with respect to the Lebesgue measure.  
Let $\gamma \in \Pi(\mu_0,\mu_1)$ be an optimal transport plan between $\mu_0$ and $\mu_1$. Then $\gamma$ does not belongs to $GMM_{2d} (\infty)$.  
\end{prop}
\begin{proof}
Since  $\mu_0$ is absolutely continuous with respect to the Lebesgue
measure, we know that the optimal transport plan is unique and is of
the form $\gamma = (\mathrm{Id},T)\#\mu_0$ for a measurable map
$T:\R^d \rightarrow \R^d$ that satisfies $T\#\mu_0=\mu_1$. Thus, if $\gamma$ belongs to
$GMM_{2d}(\infty)$,  all of its components must be degenerate Gaussian
distributions $\NN(m_k,\Sigma_k)$ such that \[\cup_{k} \left(m_k+\mathrm{Span}(\Sigma_k)\right) = \mathrm{graph}(T).\] It follows that $T$ must be affine on $\R^d$, which contradicts the hypotheses of the proposition. 
\end{proof}

When $\mu_0$ is not absolutely continuous with respect to the Lebesgue measure (which means that one of its components is degenerate), we cannot write $\gamma$ under the form~\eqref{eq:transport_map}, but we conjecture that the previous result usually still holds. A notable exception is the case where all Gaussian components of $\mu_0$ and $\mu_1$ are Dirac masses on $\R^d$, in which case $\gamma$ is also a GMM composed of Dirac masses on $\R^{2d}$.  
\smallskip

We conjecture that since optimal plans $\gamma$ between two GMM are usually not GMM, the
barycenters $(\mathrm{P}_t)\#\gamma$ between $\mu_0$ and $\mu_1$ are
also usually not GMM either (with the exception of $t=0,1$). Take the
one dimensional example of
$\mu_0 = \mathcal{N}(0,1)$ and $\mu_1 = \frac 1 2 (\delta_{-1}+\delta_1)$.
Clearly, an optimal transport map between $\mu_0$ and $\mu_1$ is
defined as
$T(x) = \mathrm{sign}(x)$.
For $t\in(0,1)$, if we denote by $\mu_t$ the barycenter between $\mu_0$ with
weight $1-t$ and $\mu_1$ with weight $t$, then it is easy to show that
$\mu_t$ has a density
$$f_t(x) = \frac{1}{1-t} \left(g\left(
    \frac{x+t}{1-t}\right)\mathbf{1}_{x<-t}+ g\left(
    \frac{x-t}{1-t}\right)\mathbf{1}_{x>t} \right),$$
where $g$ is the density of $\NN(0,1)$. The density $f_t$ is equal to $0$ on the
interval $(-t,t)$ and therefore cannot be the density of a GMM.  

\section{$MW_2$: a distance between Gaussian Mixture Models}
\label{sec:MW2}

In this section, we define a Wasserstein-type distance between
Gaussian mixtures ensuring that barycenters between Gaussian mixtures
remain Gaussian mixtures. To this aim, we restrict the set of
admissible transport plans to Gaussian mixtures and show that the
problem is well defined. Thanks to the identifiability results proved in the previous section, we will show that the corresponding optimization problem boils down to a very simple discrete formulation. 

\subsection{Definition of $MW_2$}
\begin{defi}
\label{def:MW2}
Let $\mu_0$ and $\mu_1$ be two Gaussian mixtures. We define 
  \begin{equation}
    \label{eq:MW2}
 MW_2^2(\mu_0,\mu_1) := \inf_{\gamma \in \Pi(\mu_0,\mu_1) \cap GMM_{2d}(\infty)} \int_{\R^d\times\R^d} \|y_0-y_1\|^2d\gamma(y_0,y_1).
  \end{equation}
\end{defi}
First, observe that the problem is well defined since
$\Pi(\mu_0,\mu_1) \cap GMM_{2d}(\infty)$ contains at least the product
measure $\mu_0 \otimes \mu_1$. Notice also that from the definition we
directly have that
$$MW_2(\mu_0,\mu_1) \geq W_2(\mu_0,\mu_1).$$

\subsection{An equivalent discrete formulation}

Now, we can show that this optimisation problem has a very simple discrete formulation. For $\pi_0  \in \Gamma_{K_0}$ and $\pi_1 \in \Gamma_{K_1}$, we denote by $\Pi(\pi_0,\pi_1)$ the subset of the simplex $\Gamma_{K_0\times K_1}$ with marginals $\pi_0$ and $\pi_1$, {\em i.e.}
\begin{align}
  \label{eq:Pi}
  \Pi(\pi_0,\pi_1) &= \{w \in \MM_{K_0,K_1}(\R^+);\;\;w\1_{K_1} = \pi_0;\;\;w^t\mathbf{1}_{K_0} = \pi_1\} \\
&= \{w \in \MM_{K_0,K_1}(\R^+);\;\forall k,\sum_j w_{kj} = \pi_0^k \text{ and }\forall j,\;\sum_k w_{kj} = \pi_1^j\;\}.
\end{align}

\begin{prop}
\label{prop:MW2}
 Let $\mu_0 =\sum_{k=1}^{K_0} \pi_0^k \mu_0^k$ and $\mu_1= \sum_{k=1}^{K_1} \pi_1^k \mu_1^k$ be two Gaussian mixtures, then 
  \begin{equation}
    \label{eq:DW_conj}
 MW_2^2(\mu_0,\mu_1) =  \min_{w \in \Pi(\pi_0,\pi_1)} \sum_{k,l}w_{kl} W_2^2(\mu_0^k,\mu_1^l). 
  \end{equation}
Moreover, if $w^*$ is a minimizer of~\eqref{eq:DW_conj}, and if $T_{k,l}$ is the $W_2$-optimal map between $\mu_0^k$ and $\mu_1^l$, then $\gamma^*$ defined as
$$  \gamma^* (x,y) = \sum_{k,l} w_{k,l}^\ast \,
g_{m_0^k,\Sigma_0^k}(x)\, \delta_{y=T_{k,l}(x)}$$
is a minimizer of~\eqref{eq:MW2}.
\end{prop}
\begin{proof}
First, let $w^*$ be a solution of the discrete linear program 
\begin{equation}
 \inf_{w \in \Pi(\pi_0,\pi_1)} \sum_{k,l}w_{kl} W_2^2(\mu_0^k,\mu_1^l). 
\end{equation}
For each pair $(k,l)$, let 
\[\gamma_{kl} = \mathrm{argmin}_{\gamma \in \Pi(\mu_0^k,\mu_1^l)} \int_{\R^d\times\R^d} \|y_0-y_1\|^2d\gamma(y_0,y_1)\]
 and  
\[\gamma^* = \sum_{k,l} w^*_{kl}\gamma_{kl}.\]   
Clearly, $\gamma^* \in \Pi(\mu_0,\mu_1) \cap GMM_{2d}(K_0K_1)$.
It follows that 
\begin{eqnarray*}
\sum_{k,l} w^*_{kl}W_2^2(\mu_0^k,\mu_1^l) &=& \int_{\R^d\times\R^d} \|y_0-y_1\|^2d\gamma^*(y_0,y_1)\\ 
&\geq& \min_{\gamma \in \Pi(\mu_0,\mu_1) \cap GMM_{2d}(K_0K_1)} \int_{\R^d\times\R^d} \|y_0-y_1\|^2d\gamma(y_0,y_1)\\
&\geq& \min_{\gamma \in \Pi(\mu_0,\mu_1) \cap GMM_{2d}(\infty)} \int_{\R^d\times\R^d} \|y_0-y_1\|^2d\gamma(y_0,y_1),
\end{eqnarray*}
because $GMM_{2d}(K_0K_1) \subset GMM_{2d}(\infty)$.

Now, let $\gamma$ be any element of $\Pi(\mu_0,\mu_1) \cap GMM_{2d}(\infty)$. Since $\gamma$ belongs to $GMM_{2d}(\infty)$, there exists an integer $K$ such that $\gamma = \sum_{j=1}^{K}w_j \gamma_j$. Since $\mathrm{P}_0{\#} \gamma = \mu_0$, it follows that  
\[\sum_{j=1}^{K}w_j \mathrm{P}_0{\#} \gamma_j = \sum_{k=1}^{K_0} \pi_0^k\mu_0^k.\]
Thanks to the identifiability property shown in the previous section,
we know that these two Gaussian mixtures must have the same
components, so for each $j$ in $\{1,\dots K\}$, there is $1\leq k\leq
K_0$ such that  $\mathrm{P}_0{\#} \gamma_j  = \mu_0^k$. In the same
way, there is  $1\leq l\leq K_1$ such that  $\mathrm{P}_1{\#} \gamma_j
= \mu_1^l$. It follows that $\gamma_j$ belongs to
$\Pi(\mu_0^k,\mu_1^l)$. We conclude that the mixture $\gamma$ can be
written  as a mixture of Gaussian components $\gamma_{kl} \in \Pi(\mu_0^k,\mu_1^l)$, {\em i.e} $\gamma = \sum_{k=1}^{K_0}\sum_{l=1}^{K_1}w_{kl} \gamma_{kl}$. Since $\mathrm{P}_0{\#} \gamma  = \mu_0$ and $\mathrm{P}_1{\#} \gamma  = \mu_1$, we know that $w \in \Pi(\pi_0,\pi_1)$. As a consequence,
\begin{eqnarray*}
  \int_{\R^d\times\R^d} \|y_0-y_1\|^2d\gamma(y_0,y_1) &\geq& \sum_{k=1}^{K_0}\sum_{l=1}^{K_1}w_{kl} W_2^2(\mu_0^k,\mu_1^l)\geq \sum_{k=1}^{K_0}\sum_{l=1}^{K_1}w^*_{kl} W_2^2(\mu_0^k,\mu_1^l). 
\end{eqnarray*}
This inequality holds for any $\gamma$ in $\Pi(\mu_0,\mu_1) \cap GMM_{2d}(\infty)$, which concludes the proof.
\end{proof}

It happens that the discrete form~\eqref{eq:DW_conj}, which can be
seen as an aggregation of simple Wasserstein distances between Gaussians,  has been recently proposed as an
ingenious  alternative to $W_2$ in the machine learning literature, both in
\cite{chen2016distance,chen2019aggregated} and
\cite{chen2017optimal,chen2019optimal}.
Observe that the point of view followed here in our paper is quite
different from these works, since $MW_2$ is defined in a completely continuous setting as an
optimal transport between GMMs with a restriction on couplings,
following the same kind of approach as in~\cite{bionnadal2019}. The
fact that this restriction leads to an explicit discrete formula, the
same as the one proposed independently in
\cite{chen2016distance,chen2019aggregated} and
\cite{chen2017optimal,chen2019optimal}, is quite striking.  Observe
also that thanks to the ``identifiability property'' of GMMs, this continuous formulation~\eqref{eq:MW2}  is obviously non ambiguous,
in the sense that the value of the minimium is the same whatever the
parametrization of the Gaussian mixtures $\mu_0$ and $\mu_1$. This was
not obvious from the discrete versions. We will see in
the following sections how this continuous formulation can be extended
to multi-marginal and barycenter formulations, and how it can be
generalized or used in the case of more general distributions.

Notice that we do not use in the definition and in the proof the
fact that the ground cost is quadratic. Definition~\ref{def:MW2} can
thus be generalized to other cost functions $c:\R^{2d}\mapsto
\R$. The reason why we focus on the quadratic cost is that optimal
transport plans between Gaussian measures for $W_2$ can be computed
explicitely.   
It follows from the equivalence between the continuous and discrete forms of $MW_2$ that the solution of~\eqref{eq:MW2} is very easy to compute in practice.
Another consequence of this equivalence is that there exists at least one optimal plan $\gamma^*$ for~\eqref{eq:MW2} containing less than $K_0+K_1-1$ Gaussian components. 

\begin{Corollary}
 Let $\mu_0 =\sum_{k=1}^{K_0} \pi_0^k \mu_0^k$ and $\mu_1= \sum_{k=1}^{K_1} \pi_1^k \mu_1^k$ be two Gaussian mixtures on $\R^d$, then the infimum in \eqref{eq:MW2} is attained for a given $\gamma^* \in \Pi(\mu_0,\mu_1) \cap GMM_{2d}(K_0+K_1-1)$. 
\end{Corollary}
\begin{proof}
This follows directly from the proof that there exists at least one optimal $w^*$ for~\eqref{eq:MW2} containing less than $K_0+K_1-1$ Gaussian components (see~\cite{peyre2019computational}). 
\end{proof}

\subsection{An example in one dimension}
\label{sec:example1}
In order to illustrate the behavior of the optimal maps for $MW_2$, we
focus here on a very simple example in one dimension, where $\mu_0$
and $\mu_1$ are the following mixtures of two Gaussian components
\[\mu_0 = 0.3 \NN(0.2,0.03^2)+ 0.7\NN(0.4,0.04^2),\]
\[\mu_1 = 0.6 \NN(0.6,0.06^2)+ 0.4\NN(0.8,0.07^2).\]
Figure~\ref{1D_example_plans} shows the optimal transport plans between $\mu_0$ (in blue) and $\mu_1$ (in red), both for the Wasserstein distance $W_2$ and for $MW_2$. As we can observe, the optimal transport plan for $MW_2$ (a probability measure on $\R\times \R$) is a mixture of three degenerate Gaussians measures supported by 1D lines.    

\begin{figure}[h]
  \centering
  \includegraphics[width=6cm]{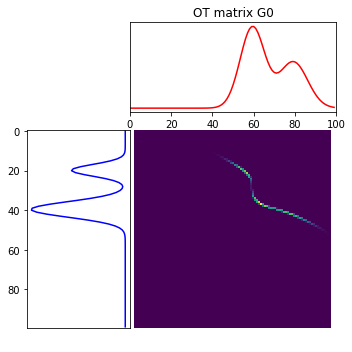}
  \includegraphics[width=6cm]{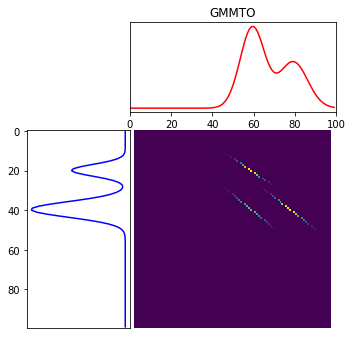}
  \caption{Transport plans between two mixtures of Gaussians $\mu_0$
    (in blue) and $\mu_1$ (in red). Left, optimal transport plan for
    $W_2$. Right, optimal transport plan for $MW_2$. (The values on
    the x-axes have been mutiplied by $100$). These examples
    have been computed using the Python Optimal Transport (POT) library \cite{flamary2017pot}.
\label{1D_example_plans}}
\end{figure}

\subsection{Metric properties of $MW_2$ and displacement interpolation}

\subsubsection{Metric properties of $MW_2$}

\begin{prop}
\label{prop:metric}
  $MW_2$ defines a metric on $GMM_d(\infty)$ and the space $GMM_{d}(\infty)$ equipped with the distance $MW_2$ is a
geodesic space. 
\end{prop}
This proposition can be proved very easily by making use of the
discrete formulation~\eqref{eq:DW_conj} of the distance (see for
instance~\cite{chen2019optimal}). For the sake of completeness, we
provide in the following a proof of the proposition using only the
continuous formulation of $MW_2$. 

\begin{proof}
First, observe that $MW_2$ is obviously symmetric and positive.
It is also clear that for any Gaussian mixture $\mu$,
$MW_2(\mu,\mu)=0$. Conversely, assume that $MW_2(\mu_0,\mu_1)=0$, it
implies that $W_2(\mu_0,\mu_1) = 0$ and thus $\mu_0 =\mu_1$ since
$W_2$ is a distance.  

It  remains to show that $MW_2$ satisfies the triangle
inequality. This is a classical consequence of the gluing lemma, but
we must be careful to check that the constructed measure remains a Gaussian mixture. Let $\mu_0,\mu_1, \mu_2$ be three Gaussian mixtures on $\R^d$.
Let $\gamma_{01}$ and  $\gamma_{12}$ be optimal plans respectively for
$(\mu_0,\mu_1)$ and  $(\mu_1,\mu_2)$ for the problem $MW_2$ (which
means that $\gamma_{01}$ and  $\gamma_{12}$  are both GMM on
$\R^{2d}$). The classical gluing lemma consists in disintegrating
$\gamma_{01}$ and $\gamma_{12}$ into  
 \[d\gamma_{01}(y_0,y_1) = d\gamma_{01}(y_0|y_1)d\mu_1(y_1) \quad
 \text{and} \quad d\gamma_{12}(y_1,y_2) =
 d\gamma_{12}(y_2|y_1)d\mu_1(y_1),\] 
 and to define
 \[d\gamma_{012}(y_0,y_1,y_2) = d\gamma_{01}(y_0|y_1)d\mu_1(y_1)d\gamma_{12}(y_2|y_1),\]
which boils down to assume independence conditionnally to the value of $y_1$. 
 Since $\gamma_{01}$ and  $\gamma_{12}$ are Gaussian mixtures on
 $\R^{2d}$, the conditional distributions $d\gamma_{01}(y_0|y_1)$ and
 $d\gamma_{12}(y_2|y_1)$ are also Gaussian mixtures for all $y_1$ in
 the support of $\mu_1$ (recalling that $\mu_1$ is the 
 marginal on $y_1$ of both $\gamma_{01}$ and $\gamma_{12}$). If we
 define a distribution $\gamma_{02}$ by integrating $\gamma_{012}$
 over the variable $y_1$, {\it i.e.} 
 \[d\gamma_{02}(y_0,y_2) = \int_{y_1 \in \R^d}
 d\gamma_{012}(y_0,y_1,y_2) = \int_{y_1 \in
   \mathrm{Supp}(\mu_1)}d\gamma_{01}(y_0|y_1)d\mu_1(y_1)d\gamma_{12}(y_2|y_1) \] 
then $\gamma_{02}$ is obviously also a Gaussian mixture on $\R^{2d}$
with marginals $\mu_0$ and $\mu_2$. 
The rest of the proof is classical. Indeed, we can write 
\begin{align*}
MW_2^2(\mu_0,\mu_2) &\leq \int_{\R^d\times \R^d} \|y_0-y_2\|^2d\gamma_{02}(y_0,y_2)=  \int_{\R^d\times \R^d\times\R^d} \|y_0-y_2\|^2d\gamma_{012}(y_0,y_1,y_2).
\end{align*}
Writing $\|y_0 - y_2\|^2 = \|y_0-y_1\|^2 + \|y_1-y_2\|^2 +  2\langle y_0-y_1,y_1-y_2\rangle $ (with $\langle \;,\;\rangle $ the Euclidean scalar product on $\R^d$), and using the Cauchy-Schwarz inequality, it follows that
\begin{align*}
MW_2^2(\mu_0,\mu_2) 
&  \le\left(\sqrt{\int_{\R^{2d}} \|y_0-y_1\|^2 d\gamma_{01}(y_0,y_1)}
  +\sqrt{ \int_{\R^{2d}} \|y_1-y_2\|^2 d\gamma_{12}(y_1,y_2)}\right)^2  .
\end{align*}
The triangle inequality follows by taking for $\gamma_{01}$ (resp. $\gamma_{12}$) the optimal plan for $MW_2$ between $\mu_0$ and $\mu_1$ (resp. $\mu_1$ and $\mu_2$).

\medskip

Now, let us show that $GMM_{d}(\infty)$ equipped with the distance
$MW_2$ is a geodesic space. 
For a path $\rho=(\rho_t)_{t\in[0,1]}$ in $GMM_d(\infty)$ (meaning that each
$\rho_t$ is a GMM on $\R^d$), we can define its length for $MW_2$ by
$$ \mathrm{Len}_{MW_2}(\rho) = \mathrm{Sup}_{N; 0=t_0\leq t_1 ... \leq t_N=1}
\quad \sum_{i=1}^N MW_2(\rho_{t_{i-1}}, \rho_{t_i})  \in [0,+\infty]
.$$
Let $\mu_0=\sum_k \pi_0^k \mu_0^k$ and
$\mu_1=\sum_l \pi_1^l \mu_1^l$ be two GMM.
Since $MW_2$ satifies the
triangle inequality, we always have that $\mathrm{Len}_{MW_2}(\rho)
\geq MW_2(\mu_0,\mu_1)$ for all paths $\rho$ such that $\rho_0=\mu_0$
and $\rho_1=\mu_1$.
To prove that $(GMM_d(\infty),MW_2)$ is a geodesic space we just have to
exhibit a path $\rho$ connecting $\mu_0$ to $\mu_1$ and such that its
length is equal to $MW_2(\mu_0,\mu_1)$.

We write $\gamma^{\ast}$ the optimal transport plan between $\mu_0$
and $\mu_1$. For $t\in(0,1)$ we can define $$\mu_t=(\rmP_t)\#\gamma^{\ast}.$$
Let $t<s \in [0,1]$ and define
$\gamma_{t,s}^{\ast}=(\rmP_t,\rmP_s)\#\gamma^{\ast}$. Then
$\gamma_{t,s}^{\ast}\in\Pi(\mu_t,\mu_s)\cap GMM_{2d}(\infty)$ and therefore
\begin{align*}
MW_2(\mu_t,\mu_s)^2  & = 
  \min_{\widetilde{\gamma}\in\Pi(\mu_t,\mu_s)\cap GMM_{2d}(\infty)} \iint \|
  y_0-y_1\|^2 \, d\widetilde{\gamma}(y_0,y_1) \\
& \leq  \iint \|  y_0-y_1\|^2 \, d\gamma_{t,s}^{\ast}(y_0,y_1) = \iint \|
         \rmP_t(y_0,y_1) - \rmP_s(y_0,y_1)\|^2 \, d\gamma^{\ast}(y_0,y_1) \\
& = \iint \| (1-t)y_0 + ty_1-(1-s)y_0-sy_1\|^2 \, d\gamma^{\ast}(y_0,y_1) \\
&=  (s-t)^2 MW_2(\mu_0,\mu_1)^2 .
\end{align*}

Thus we have that $MW_2(\mu_t,\mu_s)\leq (s-t) MW_2(\mu_0,\mu_1)$
Now, by the triangle inequality, 
\begin{align*}
MW_2(\mu_0,\mu_1) &\leq MW_2(\mu_0,\mu_t)
+MW_2(\mu_t,\mu_s)+MW_2(\mu_s,\mu_1) \\
&\leq (t+s-t+1-s) MW_2(\mu_0,\mu_1).
\end{align*}

Therefore all inequalities are equalities, and $MW_2(\mu_t,\mu_s)=
(s-t) MW_2(\mu_0,\mu_1)$ for all $0\leq t\leq s\leq 1$. This implies
that the $MW_2$ length of the path $(\mu_t)_t$ is equal to
$MW_2(\mu_0,\mu_1)$. It allows us to conclude that $(GMM_d(\infty),MW_2)$
is a geodesic space, and we have also given the explicit expression of
the geodesic.
\end{proof}

The following Corollary is a direct consequence of the previous results.
\begin{Corollary}
The {barycenters} between $\mu_0=\sum_k \pi_0^k \mu_0^k$ and
$\mu_1=\sum_l \pi_1^l \mu_1^l$ all belong to $GMM_d(\infty)$ and can be written explicitely as
$$\forall t\in[0,1] , \quad \mu_t = \rmP_t\#\gamma^{\ast} = \sum_{k,l} w^\ast_{k,l}
\mu_t^{k,l} ,$$
where $w^\ast$ is an optimal solution of~\eqref{eq:DW_conj}, and
$\mu_t^{k,l}$ is the displacement interpolation between $\mu_0^k$ and
$\mu_1^l$. When $\Sigma_0^k$ is non-singular, it is given by
$$\mu_t^{k,l} = ((1-t)\mathrm{Id}+tT_{k,l})\# \mu_0^k  ,$$
with $T_{k,l}$ the affine transport map between $\mu_0^k$ and $\mu_1^l$ given by
Equation~\eqref{eqT:eq}. These barycenters have less than $K_0+K_1-1$ components.   
\end{Corollary}

\subsubsection{1D and 2D barycenter examples}

\paragraph{One dimensional case}

\begin{figure}[h!]
\setlength{\tempwidth}{.16\linewidth}
\settoheight{\tempheight}{\includegraphics[width=\tempwidth]{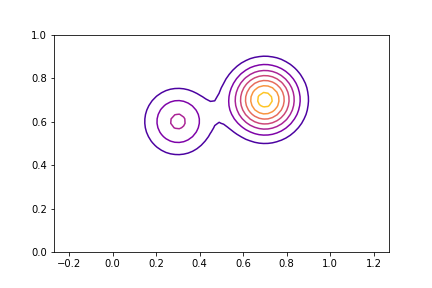}}%
\centering
\hspace{\baselineskip}
\columnname{$t=0.0$}\hfil
\columnname{$t=0.2$}\hfil
\columnname{$t=0.4$}\hfil
\columnname{$t=0.6$}\hfil
\columnname{$t=0.8$}\hfil
\columnname{$t=1.0$}\\
 \rowname{$W_2$}
 \includegraphics[width=\tempwidth]{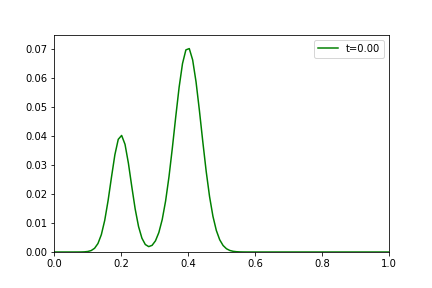} \hfil
\includegraphics[width=\tempwidth]{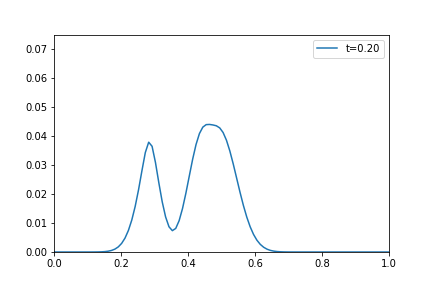}\hfil
\includegraphics[width=\tempwidth]{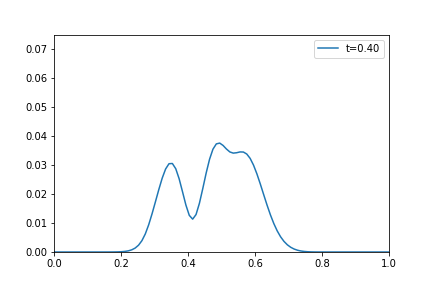}\hfil
\includegraphics[width=\tempwidth]{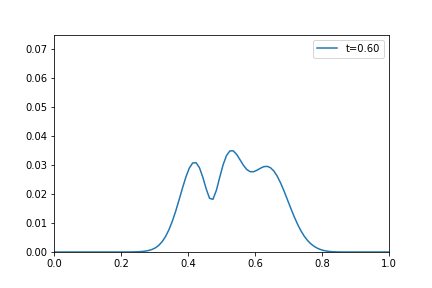}\hfil
\includegraphics[width=\tempwidth]{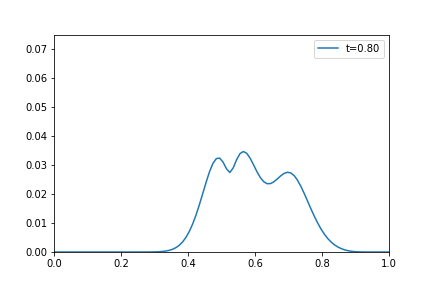}\hfil
\includegraphics[width=\tempwidth]{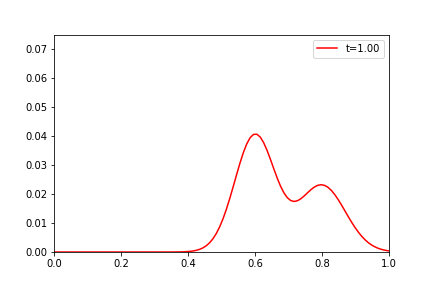}

\rowname{$MW_2$}
\includegraphics[width=\tempwidth]{barycenter1D_00.png}\hfil
\includegraphics[width=\tempwidth]{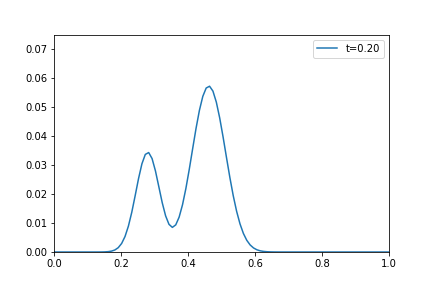}\hfil
\includegraphics[width=\tempwidth]{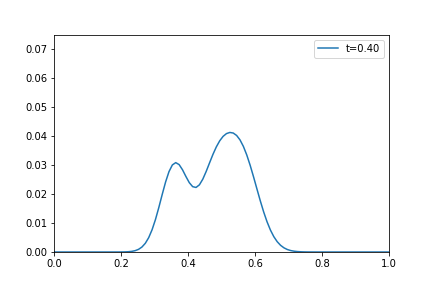}\hfil
\includegraphics[width=\tempwidth]{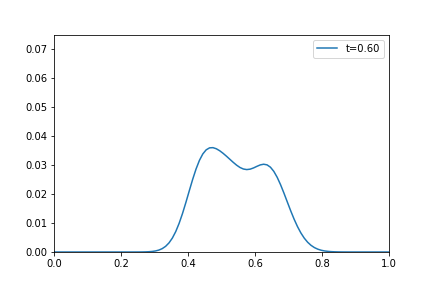}\hfil
 \includegraphics[width=\tempwidth]{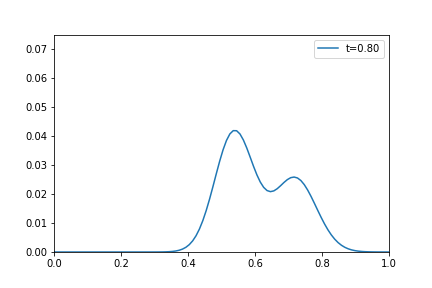}\hfil
 \includegraphics[width=\tempwidth]{barycenter1D_05.png}

  \caption{Barycenters $\mu_t$ between two Gaussian mixtures
    $\mu_0$ (green curve) and $\mu_1$ (red curve). Top: barycenters
    for the metric $W_2$. Bottom: barycenters
    for the metric $MW_2$. The barycenters are computed for $t=0.2,0.4,0.6,0.8$. \label{1D_barycenters}}
\end{figure}

Figure~\ref{1D_barycenters} shows barycenters $\mu_t$ for
$t=0.2,0.4,0.6,0.8$    
between the $\mu_0$ and $\mu_1$ defined in
Section~\ref{sec:example1}, for  both the metric $W_2$ and
$MW_2$. Observe that the barycenters computed for $MW_2$ are a bit
more regular (we know that they are mixtures of at most 3 Gaussian
components) than those obtained for $W_2$.

\paragraph{Two dimensional case}

\begin{figure}[h!]
\setlength{\tempwidth}{.16\linewidth}
\settoheight{\tempheight}{\includegraphics[width=\tempwidth]{barycenter2D_00.png}}%
\centering
\hspace{\baselineskip}
\columnname{$t=0.0$}\hfil
\columnname{$t=0.2$}\hfil
\columnname{$t=0.4$}\hfil
\columnname{$t=0.6$}\hfil
\columnname{$t=0.8$}\hfil
\columnname{$t=1.0$}\\
\centering
 \rowname{$W_2$}
 \includegraphics[width=\tempwidth]{barycenter2D_00.png} \hfil
\includegraphics[width=\tempwidth]{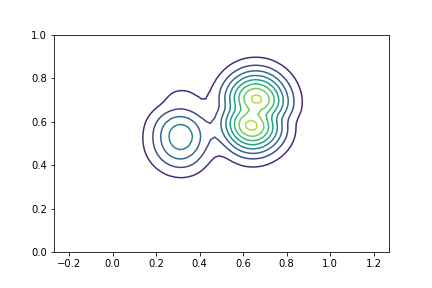}\hfil
\includegraphics[width=\tempwidth]{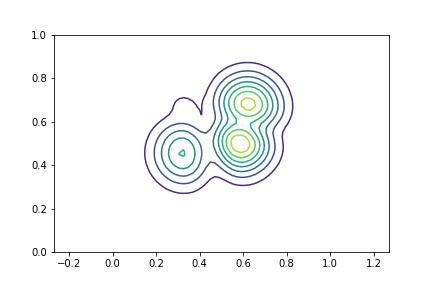}\hfil
\includegraphics[width=\tempwidth]{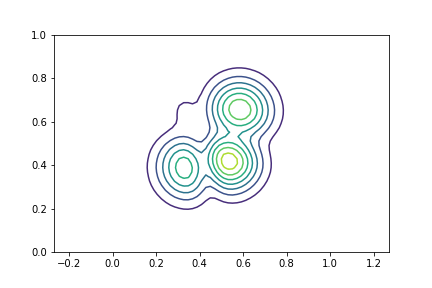}\hfil
\includegraphics[width=\tempwidth]{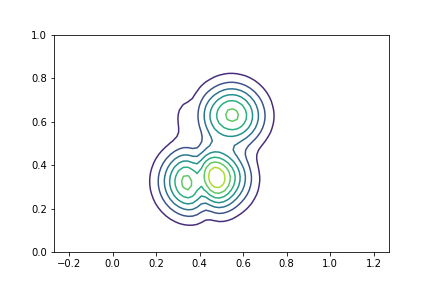}\hfil
\includegraphics[width=\tempwidth]{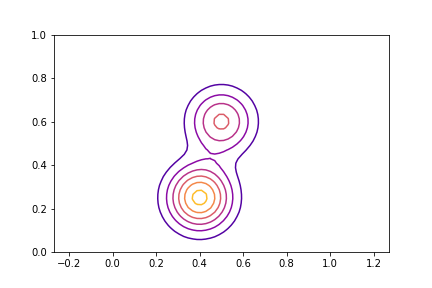}

\rowname{$MW_2$}
\includegraphics[width=\tempwidth]{barycenter2D_00.png}\hfil
\includegraphics[width=\tempwidth]{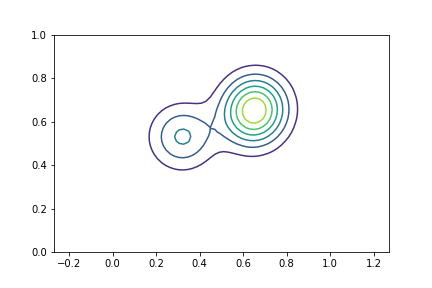}\hfil
\includegraphics[width=\tempwidth]{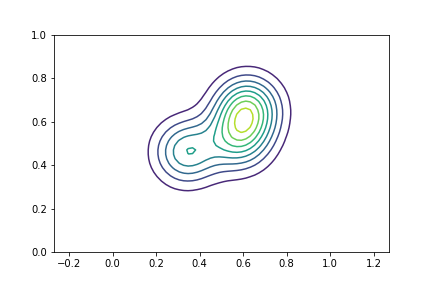}\hfil
\includegraphics[width=\tempwidth]{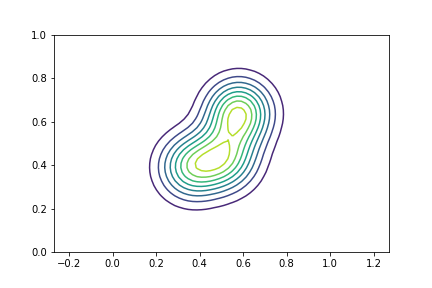}\hfil
 \includegraphics[width=\tempwidth]{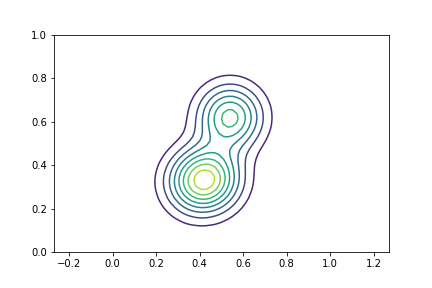}\hfil
 \includegraphics[width=\tempwidth]{barycenter2D_05.png}
 
\caption{Barycenters $\mu_t$ between two Gaussian mixtures $\mu_0$ (first column) and $\mu_1$ (last column). Top: barycenters for the metric $W_2$. Bottom: barycenters for the metric $MW_2$.\label{2D_barycenters} The barycenters are computed for $t=0.2,0.4,0.6,0.8$.}
\end{figure}

Figure~\ref{2D_barycenters} shows barycenters $\mu_t$  between the following two dimensional mixtures
\[\mu_0 = 0.3 \NN\left(
  \begin{pmatrix}
    0.3\\0.6
  \end{pmatrix}
,0.01 I_2\right)+ 0.7 \NN\left(
  \begin{pmatrix}
    0.7\\0.7
  \end{pmatrix}
,0.01 I_2\right),\]
\[\mu_1 = 0.4 \NN\left(
  \begin{pmatrix}
    0.5\\0.6
  \end{pmatrix}
,0.01 I_2\right)+ 0.6 \NN\left(
  \begin{pmatrix}
    0.4\\0.25
  \end{pmatrix}
,0.01 I_2\right),\]
where $I_2$ is the $2\times 2$ identity matrix.
Notice that the $MW_2$ geodesic looks much more regular, each barycenter is a
mixture of less than three Gaussians.

\subsection{Comparison between $MW_2$ and $W_2$}
\label{sec:comparison}

\begin{prop}
\label{prop:ineq}
Let $\mu_0 \in GMM_d(K_0)$ and $\mu_1 \in GMM_d(K_1)$ be two Gaussian mixtures, written as in~\eqref{eq:GMM}.
Then,   
\[W_2(\mu_0,\mu_1) \leq MW_2(\mu_0,\mu_1) \leq W_2(\mu_0,\mu_1) + \sum_{i=0,1}\left(2\sum_{k=1}^{K_i} \pi_i^k \mathrm{trace}(\Sigma_i^k)\right)^{\frac 1 2}.\]
The left-hand side inequality is attained when for instance
\begin{itemize}
\item $\mu_0$ and $\mu_1$ are both composed of only one Gaussian component,
\item $\mu_0$ and $\mu_1$ are finite linear combinations of Dirac
  masses, 
\item $\mu_1$ is obtained from $\mu_0$ by an affine transformation.
\end{itemize}
\end{prop}

As we already noticed it, the first inequality is obvious and follows
from the definition of $MW_2$. It might not be completely intuitive that
$MW_2$ can indeed be strictly larger than $W_2$ because of the density
property of $GMM_d(\infty)$ in $\PP_2(\R^d)$. This follows from the
fact that our optimization problem has constraints
$\gamma\in\Pi(\mu_0,\mu_1)$. Even if any measure $\gamma$ in
$\Pi(\mu_0,\mu_1)$ can be approximated by a sequence of Gaussian
mixtures, this sequence of Gaussian mixtures will generally not belong
to $\Pi(\mu_0,\mu_1)$, hence explaining the difference between $MW_2$ and $W_2$.

In order to show that $MW_2$ is always smaller than the sum of $W_2$
plus a term depending on the trace of the covariance matrices of the two Gaussian mixtures, we start with a lemma which makes more explicit the distance $MW_2$ between a Gaussian mixture and a mixture of Dirac distributions.   

\begin{lemma}
\label{lemma1}
 Let $\mu_0 = \sum_{k=1}^{K_0} \pi_0^k \mu_0^k $ with $\mu_0^k = \NN(m_0^k,\Sigma_0^k)$ and $\mu_{1} =  \sum_{k=1}^{K_1} \pi_1^k \delta_{m_1^k}$. Let $\tilde{\mu}_0 = \sum_{k=1}^{K_0} \pi_0^k \delta_{m_0^k}$ ($\tilde{\mu}_0$ only retains the means of $\mu_0$).
Then,   
\[MW_2^2(\mu_0,\mu_1) = W_2^2( \tilde{\mu}_0,\mu_1) + \sum_{k=1}^{K_0} \pi_0^k\mathrm{trace}(\Sigma_0^k).\]
\end{lemma}
\begin{proof}
  \begin{eqnarray*}
    MW_2^2(\mu_0,\mu_1) &=& \inf_{w \in \Pi(\pi_0,\pi_1)} \sum_{k,l}w_{kl} W_2^2(\mu_0^k,\delta_{m_1^l})= \inf_{w \in \Pi(\pi_0,\pi_1)} \sum_{k,l}w_{kl} \left(\|m_1^l - m_0^k\|^2 + \mathrm{trace}(\Sigma_0^k)\right) \\
&=& \inf_{w \in \Pi(\pi_0,\pi_1)} \sum_{k,l}w_{kl} \|m_1^l - m_0^k\|^2 + \sum_{k}\pi_0^k\mathrm{trace}(\Sigma_0^k) = W_2^2(\tilde{\mu}_0,\mu_1) + \sum_{k=1}^{K_0} \pi_0^k\mathrm{trace}(\Sigma_0^k).
  \end{eqnarray*}
\end{proof}

In other words, the squared distance $MW_2^2$ between $\mu_0$ and
$\mu_1$ is the sum of the squared Wasserstein distance between
$\tilde{\mu}_0$ and $\mu_1$ and a linear combination of the traces of
the covariance matrices of the components of $\mu_0$. We are now in a position to show the other inequality between $MW_2$ and $W_2$. 

\begin{proof} [Proof of Proposition~\ref{prop:ineq}] 
Let $(\mu_0^n)_n$ and $(\mu_1^n)_n$ be two sequences of mixtures of Dirac masses respectively converging  to $\mu_0$ and $\mu_1$ in $\PP_2(\R^d)$.
Since  $MW_2$ is a distance,
  \begin{eqnarray*}
    MW_2(\mu_0,\mu_1) &\leq&  MW_2(\mu_0^n,\mu_1^n)+MW_2(\mu_0,\mu_0^n)+ MW_2(\mu_1,\mu_1^n) \\
&=&  W_2(\mu_0^n,\mu_1^n)+MW_2(\mu_0,\mu_0^n)+ MW_2(\mu_1,\mu_1^n).
  \end{eqnarray*}
We study in the following the limits of these three terms when $n\rightarrow +\infty$. 

First, observe that $MW_2(\mu_0^n,\mu_1^n) = W_2(\mu_0^n,\mu_1^n) \longrightarrow_{n\rightarrow \infty}W_2(\mu_0,\mu_1) $ since $W_2$ is continuous on $\PP_2(\R^d)$.  

Second, using Lemma~\ref{lemma1}, for $i=0,1$,
$$  MW_2^2(\mu_i,\mu_i^n) =   W_2^2( \tilde{\mu}_i,\mu_i^n) + \sum_{k=1}^{K_i} \pi_i^k\mathrm{trace}(\Sigma_i^k) 
\longrightarrow_{n\rightarrow \infty}  W_2^2( \tilde{\mu}_i,\mu_i)+
\sum_{k=1}^{K_i} \pi_i^k\mathrm{trace}(\Sigma_i^k). $$

Define the measure $d\gamma(x,y) = \sum_{k=1}^{K_i} \pi_i^k \delta_{m_i^k}(y)g_{m_i^k,\Sigma_i^k}(x)dx$, with $g_{m_i^k,\Sigma_i^k}$ the probability density function of the Gaussian distribution $\NN(m_i^k,\Sigma_i^k)$. The probability measure $\gamma$ belongs to $\Pi(\mu_i,\tilde{\mu}_i)$, so 
\begin{eqnarray*}
  W_2^2( \mu_i, \tilde{\mu}_i)&\leq& \int \|x - y\|^2 d\gamma(x,y) =  \sum_{k=1}^{K_i} \pi_i^k \int_{\R^d}\|x - m_i^k\|^2 g_{m_i^k,\Sigma_i^k}(x)dx \\ &=& \sum_{k=1}^{K_i} \pi_i^k \mathrm{trace}(\Sigma_i^k).
\end{eqnarray*}
  
We conclude that 
  \begin{eqnarray*}
    MW_2(\mu_0,\mu_1) &\leq&  \lim\inf_{n\rightarrow \infty} \left( W_2(\mu_0^n,\mu_1^n)+MW_2(\mu_0,\mu_0^n)+ MW_2(\mu_1,\mu_1^n) \right)\\
&\leq&  W_2(\mu_0,\mu_1) + \left(W_2^2( \tilde{\mu}_0,\mu_0) + \sum_{k=1}^{K_0} \pi_0^k\mathrm{trace}(\Sigma_0^k)\right)^{\frac 1 2} + \left(W_2^2(\tilde{\mu}_1,\mu_1) + \sum_{k=1}^{K_1} \pi_1^k\mathrm{trace}(\Sigma_1^k)\right)^{\frac 1 2}\\
&\leq&   W_2(\mu_0,\mu_1) + \left(2\sum_{k=1}^{K_0} \pi_0^k\mathrm{trace}(\Sigma_0^k)\right)^{\frac 1 2}+ \left(2\sum_{k=1}^{K_1} \pi_1^k\mathrm{trace}(\Sigma_1^k)\right)^{\frac 1 2}.
  \end{eqnarray*}
 
This ends the proof of the proposition.

\end{proof}

Observe that if $\mu$ is a Gaussian distribution $\NN(m,\Sigma)$ and $\mu^n$ a distribution supported by a finite number of points which converges to  $\mu$ in $\PP_2(\R^d)$, then 
$$ W_2^2(\mu,\mu^n)\longrightarrow_{n\rightarrow \infty} 0$$
 and
\[MW_2(\mu,\mu^n) = \left(W_2^2(\tilde{\mu},\mu^n) + \mathrm{trace}(\Sigma)\right)^{\frac 1 2} \longrightarrow_{n\rightarrow \infty} \left(2 \mathrm{trace}(\Sigma)\right)^{\frac 1 2} \neq 0.\]

Let us also remark that if $\mu_0$ and $\mu_1$ are Gaussian mixtures such that $\max_{k,i} \mathrm{trace}(\Sigma_i^k) \leq \varepsilon$, then
\[MW_2(\mu_0,\mu_1) \leq W_2(\mu_0,\mu_1) + 2\sqrt{2\varepsilon}.\]

\subsection{Generalization to other mixture models}

A natural question is to know if the methodology we have developped
here, and that restricts the set of possible coupling
measures to Gaussian mixtures, can be extended to other families of
mixtures. Indeed, in the image processing litterature, as well as in
many other fields, mixture models beyond Gaussian ones are widely
used, such as Generalized Gaussian Mixture Models~\cite{DeledalleGGM} or
mixtures of T-distributions~\cite{van2014student}, for instance. 
Now, to extend our methodology to other mixtures, we need two main
properties: (a) the identifiability property (that will ensure that there
is a canonical way to write a distribution as a mixture); and (b) a
marginal consistency property  (we need all the marginal of an element
of the family to remain in the same family). These two properties
permit in particular to generalize the proof of
Proposition~\ref{prop:MW2}. In order to make the discrete formulation
convenient for numerical computations, we also need  that  the $W_2$ distance between any two elements of the family must be easy to compute.

Starting from this last requirement, we can consider a family of
elliptical distributions, where the elements are of the form
$$\forall x\in\R^d ,   f_{m,\Sigma} (x) = C_{h,d,\Sigma} \,  h(
(x-m)^t \Sigma^{-1} (x-m)) ,$$
where $m\in\R^d$, $\Sigma$ is a positive definite symmetric matrix and
$h$ is a given function from 
$[0,+\infty)$ to  $[0,+\infty)$. Gaussian distributions are an
example, with $h(t)=\exp(-t/2)$. Generalized Gaussian distributions are obtained with
$h(t)=\exp(-t^\beta)$, with $\beta$ not necessarily equal to
$1$. T-distributions are also in this family, with $h(t)=(1+t/\nu)^{-(\nu+d)/2}$, etc.
Thanks to their elliptical contoured property, the $W_2$ distance
between two elements in such a family (i.e. $h$ fixed) can be
explicitely computed (see Gelbrich \cite{Gelbrich1990}), and yields a
formula that is the same as the one in the Gaussian case (Equation
\eqref{eq:wasserstein_gaussian}). 
In such a family, the identifiability property can be checked, using
the asymptotic behavior in all directions of $\R^d$.
Now, if we want the marginal consistency property to be also satisfied
(which is necessary if we want the coupling restriction problem to be
well-defined), the choice of $h$ is very limited. Indeed, Kano in
\cite{Kano1994}, proved that the only elliptical distributions with
the marginal consistency property are the ones which are a scale mixture of normal
distributions with a mixing variable that is unrelated to the
dimension $d$. So, generalized Gaussian distributions don't satisfy
this marginal consistency property, but T-distributions do.

\section{Multi-marginal formulation and barycenters}
\label{sec:multimarginal}

\subsection{Multi-marginal formulation for $MW_2$}

Let $\mu_0,\mu_1\dots,\mu_{J-1}$ be $J$ Gaussian mixtures on
$\R^d$, and let $\lambda_0,\dots\lambda_{J-1}$ be $J$ positive weights
summing to $1$. The multi-marginal  version of our optimal transport
problem restricted to Gaussian mixture models can be written 

{ \begin{equation}
  \label{eq:multi-marginal_MW2}
 MMW_2^2(\mu_0,\dots,\mu_{J-1}):=\hspace{-0.2em}\inf_{\gamma \in \Pi(\mu_0,\dots,\mu_{J-1})\cap GMM_{Jd}(\infty)} \int_{\R^{dJ}}c(x_0,\dots,x_{J-1})d\gamma(x_0,\dots,x_{J-1}),
\end{equation}}
where 
\begin{equation}
c(x_0,\dots,x_{J-1}) = \sum_{i=0}^{J-1} \lambda_i\|x_i - B(x)\|^2 = \frac{1}{2} \sum_{i,j=0}^{J-1} \lambda_i\lambda_j\|x_i - x_j\|^2
\label{eq:cost_multimarginal}
\end{equation}
and where $\Pi(\mu_0,\mu_1,\dots,\mu_{J-1})$ is the set of probability
measures on $(\R^d)^{J}$ having $\mu_0$, $\mu_1$, $\dots$, $\mu_{J-1}$ as
marginals. 

Writing for every $j$, $\mu_{j} = \sum_{k=1}^{K_j} \pi_j^k\mu_j^k$,
and 
using exactly the same arguments as in Proposition~\ref{prop:MW2}, we
can easily show the following result. 

\begin{prop}
\label{prop:MMW2}
The optimisation problem~\eqref{eq:multi-marginal_MW2} can be rewritten under the discrete form
{\small \begin{equation}
  \label{eq:discrete_multi-marginal}
  MMW_2^2(\mu_0,\dots,\mu_{J-1})=\min_{w \in \Pi(\pi_0,\dots,\pi_{J-1})} \sum_{k_0,\dots,k_{J-1}=1}^{K_0,\dots,K_{J-1}}w_{k_0\dots k_{J-1}}mmW_2^2(\mu_0^{k_0},\dots,\mu_{J-1}^{k_{J-1}}),
\end{equation}}
where $\Pi(\pi_0,\pi_1,\dots,\pi_{J-1})$ is the subset of tensors $w$ in $ \MM_{K_0,K_1,\dots,K_{J-1}}(\R^+)$ having
$\pi_0$, $\pi_1$, $\dots$, $\pi_{J-1}$ as discrete marginals, {\it i.e.} such that
\begin{equation}
  \label{eq:Pi_multi}
 \forall j\in\{0,\dots,J-1\}, \;\forall k \in
  \{1,\dots,K_j\},\sum_{\substack{1\leq k_0\leq K_0 \\\dots\\1\leq
  k_{j-1}\leq K_{j-1}\\k_j = k\\1\leq k_{j+1}\leq K_{j+1}\\\dots \\1\leq
  k_{J-1}\leq K_{J-1}}} w_{k_0k_1\dots k_{J-1}} = \pi_j^k.
\end{equation}
  Moreover, the  solution $\gamma^*$ of~\eqref{eq:multi-marginal_MW2} can be written
\begin{equation}
\gamma^* = \sum_{\substack{1\leq k_0 \leq K_0\\ \dots \\1\leq k_{J-1}
    \leq K_{J-1}}} w^*_{k_0 k_1\dots
  k_{J-1}}\gamma^*_{k_0 k_1\dots k_{J-1}},
\end{equation}
where $w^*$ is solution of~\eqref{eq:discrete_multi-marginal}
and $\gamma^*_{k_0 k_1\dots k_{J-1}}$ is the optimal multi-marginal plan between the
Gaussian measures $\mu_0^{k_0},\dots,\mu_{J-1}^{k_{J-1}}$ (see Section~\ref{sec:multimarginal_gaussian}).
\end{prop}

From Section~\ref{sec:multimarginal_gaussian}, we know how to construct the optimal multi-marginal plans  $\gamma^*_{k_0 k_1\dots k_{J-1}}$, which means that computing a solution for~\eqref{eq:multi-marginal_MW2} boils down to solve the linear program~\eqref{eq:discrete_multi-marginal} in order to find $w^*$. 

\subsection{Link with the $MW_2$-barycenters}
We will now show the link between the previous multi-marginal problem and the barycenters for $MW_2$. 
\begin{prop}
  The barycenter problem
  \begin{equation}
    \label{eq:MW2_bary}
    \inf_{\nu \in GMM_d(\infty)} \sum_{j=0}^{J-1} \lambda_j MW_2^2(\mu_j,\nu),
  \end{equation}
has a solution given by $\nu^* = B\# \gamma^*$, where $\gamma^*$ is an
optimal plan for the multi-marginal problem~\eqref{eq:multi-marginal_MW2}.    
\end{prop}
\begin{proof}
For any $\gamma \in \Pi(\mu_0,\dots,\mu_{J-1})\cap GMM_{Jd}(\infty)$,  we define $\gamma_j = (P_j,B)\# \gamma$, with $B$ the barycenter application defined in \eqref{eq:barycenter} and $P_j:(\R^d)^J\mapsto \R^d$ such that $P(x_0,\dots,x_{J-1}) = x_j$. Observe that $\gamma_j$ belongs to $\Pi(\mu_j,\nu)$ with $\nu = B\#\gamma$. The probability measure $\gamma_j$ also belongs to $GMM_{2d}(\infty)$ since $(P_j,B)$ is a linear application. It follows that
\begin{align*}
\int_{(\R^d)^J}   \sum_{j=0}^{J-1} \lambda_j \|x_j - B(x)\|^2 d\gamma(x_0,\dots,x_{J-1})  &= \sum_{j=0}^{J-1} \lambda_j \int_{(\R^d)^J}   \|x_j - B(x)\|^2 d\gamma(x_0,\dots,x_{J-1})  \\
&= \sum_{j=0}^{J-1} \lambda_j \int_{\R^d\times \R^d}   \|x_j - y\|^2 d\gamma_j(x_j,y) \\
&\ge \sum_{j=0}^{J-1} \lambda_j MW_2^2(\mu_j,\nu).
\end{align*}
This inequality holds for any arbitrary $\gamma \in \Pi(\mu_0,\dots,\mu_{J-1})\cap GMM_{Jd}(\infty)$, thus 
\[MMW_2^2(\mu_0,\dots,\mu_{J-1}) \ge \inf_{\nu \in GMM_d(\infty)} \sum_{j=0}^{J-1} \lambda_j MW_2^2(\mu_j,\nu).\]

Conversely, for any $\nu$ in $GMM_{d}(\infty)$, we can write $\nu = \sum_{l=1}^L \pi_{\nu}^l \nu^l$, the $\nu^l$ being Gaussian probability measures. We also write $\mu_j = \sum_{k=1}^{K_j} \pi_j^k\mu_j^k$, and we call $w^j$ the optimal discrete plan for $MW_2$  between the mixtures $\mu_j$ and $\nu$ (see Equation~\eqref{eq:DW_conj}). Then, 
\begin{align*}
\sum_{j=0}^{J-1} \lambda_j MW_2^2(\mu_j,\nu) &= \sum_{j=0}^{J-1} \lambda_j \sum_{k,l} w_{k,l}^j W_2^2(\mu_j^k,\nu^l).
\end{align*}
Now, if we define a $K_0\times \dots \times K_{J-1}\times L$ tensor $\alpha$ and a $K_0\times \dots \times K_{J-1}$ tensor $\overline{\alpha}$ by
\[\alpha_{k_0\dots k_{J-1} l} = \frac{\prod_{j=0}^{J-1} w_{k_j,l}^j}{(\pi_{\nu}^l)^{J-1}}\;\; \text{ and }\quad \overline{\alpha}_{k_0\dots k_{J-1}} = \sum_{l=1}^L \alpha_{k_0\dots k_{J-1} l},\]
clearly $\alpha \in \Pi(\pi_0,\dots,\pi_{J-1},\pi_{\nu})$ and
$\overline{\alpha} \in \Pi(\pi_0,\dots,\pi_{J-1})$. 
Moreover,
\begin{align*}
\sum_{j=0}^{J-1} \lambda_j MW_2^2(\mu_j,\nu) & = \sum_{j=0}^{J-1} \lambda_j \sum_{k_j=1}^{K_j}\sum_{l=1}^L w_{k_j,l}^j W_2^2(\mu_j^{k_j},\nu^l)\\
&= \sum_{j=0}^{J-1} \lambda_j \sum_{k_0,\dots,k_{J-1},l} \alpha_{k_0\dots k_{J-1} l} W_2^2(\mu_j^{k_j},\nu^l)\\
&=  \sum_{k_0,\dots,k_{J-1},l} \alpha_{k_0\dots k_{J-1} l} \sum_{j=0}^{J-1} \lambda_j W_2^2(\mu_j^{k_j},\nu^l)\\
& \ge  \sum_{k_0,\dots,k_{J-1},l} \alpha_{k_0\dots k_{J-1} l} mmW_2^2(\mu_0^{k_0},\dots,\mu_{J-1}^{k_{J-1}}) \quad \text{ (see Equation~\eqref{eq:MW2_bary})}\\
& = \sum_{k_0,\dots,k_{J-1}} \overline{\alpha}_{k_0\dots k_{J-1}}
  mmW_2^2(\mu_0^{k_0},\dots,\mu_{J-1}^{k_{J-1}}) \ge MMW_2^2(\mu_0,\dots,\mu_{J-1}),
\end{align*}
the last inequality being a consequence of Proposition~\ref{prop:MMW2}.
Since this holds for any arbitrary $\nu$ in $GMM_{d}(\infty)$, this ends the proof.
\end{proof}

The following corollary gives a more explicit formulation for the barycenters for $MW_2$, and shows that the number of Gaussian components in the mixture is much smaller than $\prod_{j=0}^{J-1} K_j$. 

\begin{Corollary}
 Let $\mu_0,\dots,\mu_{J-1}$ be $J$ Gaussian mixtures such that all
 the involved covariance matrices are positive definite, then the solution of~\eqref{eq:MW2_bary} can be written 
 \begin{equation}
   \label{eq:barycenterMW2}
   \nu = \sum_{k_0,\dots,k_{J-1}} w^*_{k_0\dots k_{J-1}} \nu_{k_0 \dots k_{J-1}}
 \end{equation}
where $\nu_{k_0\dots k_{J-1}}$ is the Gaussian barycenter for $W_2$
between the components $\mu_0^{k_0},\dots, \mu_{J-1}^{k_{J-1}}$, and
$w^*$ is the optimal solution of~\eqref{eq:discrete_multi-marginal}.   
Moreover, this barycenter has less than $K_0+\dots+K_{J-1}-J+1$
non-zero coefficients. 
\end{Corollary}
\begin{proof}
This follows directly from the proof of the previous propositions. The linear program~\eqref{eq:discrete_multi-marginal} has  $K_0+\dots+K_{J-1}-J+1$ affine constraints, and thus must have at least a solution with  less than $K_0+\dots+K_{J-1}-J+1$ components.
\end{proof}

To conclude this section, it is important to emphasize that the problem of barycenters for the distance $MW_2$, as defined in~\eqref{eq:MW2_bary}, is completely different from 
 \begin{equation}
    \label{eq:GMM_W2_bary}
    \inf_{\nu \in GMM_d(\infty)} \sum_{j=0}^{J-1} \lambda_j W_2^2(\mu_j,\nu).
  \end{equation}
Indeed, since $GMM_d(\infty)$ is dense in $\PP_2(\R^d)$ and the total
cost on the right is continuous on $\PP_2(\R^d)$, the infimum
in~\eqref{eq:GMM_W2_bary} is exactly the same as the infimum over
$\PP_2(\R^d)$. Even if the barycenter for $W_2$ is not a mixture
itself, it can be approximated by a sequence of  Gaussian mixtures
with any desired precision. Of course, these mixtures might have a
very high number of components in practice.

\subsection{Some examples}

The previous propositions give us a very simple way to compute
barycenters between Gaussian mixtures for the metric $MW_2$. For given
mixtures $\mu_0,\dots,\mu_{J-1}$,  we first compute all the values
$mmW_2(\mu_0^{k_0},\dots ,\mu_{J-1}^{k_{J-1}})$ between their
components (and these values can be computed iteratively, see
Section~\ref{sec:multimarginal_gaussian}) and the corresponding
Gaussian barycenters $\nu_{k_0\dots k_{J-1}}$. Then we solve the
linear program~\eqref{eq:discrete_multi-marginal} to find $w^*$.

Figure~\ref{bary_4GMM} shows the barycenters between the following simple two dimensional mixtures
{\small
  \begin{align*}
\mu_0 =& \frac 1 3 \NN\left(
  \begin{pmatrix}
    0.5\\0.75
  \end{pmatrix}
,0.025 
 \begin{pmatrix}
    0.1 & 0 \\0 & 0.05
  \end{pmatrix}\right)
+ \frac 1 3 \NN\left(
  \begin{pmatrix}
    0.5\\0.25
  \end{pmatrix}
,0.025 
 \begin{pmatrix}
    0.1 & 0 \\0 & 0.05
  \end{pmatrix}\right)\\&+\frac 1 3 \NN\left(
  \begin{pmatrix}
    0.5\\0.5
  \end{pmatrix}
,0.025 
 \begin{pmatrix}
    0.06 & 0 \\0.05 & 0.05
  \end{pmatrix}\right),\\
 \mu_1 =& \frac 1 4 \NN\left(
  \begin{pmatrix}
    0.25\\0.25
  \end{pmatrix}
,0.01 I_2 \right)
+ \frac 1 4 \NN\left(
  \begin{pmatrix}
    0.75\\0.75
  \end{pmatrix}
,0.01 I_2\right)
+\frac 1 4 \NN\left(
  \begin{pmatrix}
    0.7\\0.25
  \end{pmatrix}
,0.01 I_2\right)\\
&+\frac 1 4 \NN\left(
  \begin{pmatrix}
    0.25\\0.75
  \end{pmatrix}
,0.01 I_2\right), \\ 
\mu_2 =& \frac 1 4 \NN\left(
  \begin{pmatrix}
    0.5\\0.75
  \end{pmatrix}
,0.025 
 \begin{pmatrix}
    1 & 0 \\0 & 0.05
  \end{pmatrix}\right)
+ \frac 1 4 \NN\left(
  \begin{pmatrix}
    0.5\\0.25
  \end{pmatrix}
,0.025 
 \begin{pmatrix}
    1 & 0 \\0 & 0.05
  \end{pmatrix}\right)\\
&+\frac 1 4 \NN\left(
  \begin{pmatrix}
    0.25\\0.5
  \end{pmatrix}
,0.025 
 \begin{pmatrix}
    0.05 & 0 \\0 & 1
  \end{pmatrix}\right)+\frac 1 4 \NN\left(
  \begin{pmatrix}
    0.75\\0.5
  \end{pmatrix}
,0.025 
 \begin{pmatrix}
    0.05 & 0 \\0 & 1
  \end{pmatrix}\right),\\
\mu_3 =& \frac 1 3 \NN\left(
  \begin{pmatrix}
    0.8\\0.7
  \end{pmatrix}
,0.01 
 \begin{pmatrix}
    2 & 0 \\1 & 1
  \end{pmatrix}\right)
+ \frac 1 3 \NN\left(
  \begin{pmatrix}
    0.2\\0.7
  \end{pmatrix}
,0.01 
 \begin{pmatrix}
    2 & 0 \\-1 & 1
  \end{pmatrix}\right)\\
&+\frac 1 3 \NN\left(
  \begin{pmatrix}
    0.5\\0.3
  \end{pmatrix}
,0.01 
 \begin{pmatrix}
    6 & 0 \\0 & 1
  \end{pmatrix}\right),
 \end{align*} 
}
where $I_2$ is the $2\times 2$ identity matrix. Each barycenter is a mixture of at most $K_0+K_1+K_2+K_3 -4 + 1 = 11$ components. By thresholding the mixtures densities, this yields barycenters between 2-D shapes.

\begin{figure}[h]
  \centering
  \includegraphics[width=6cm]{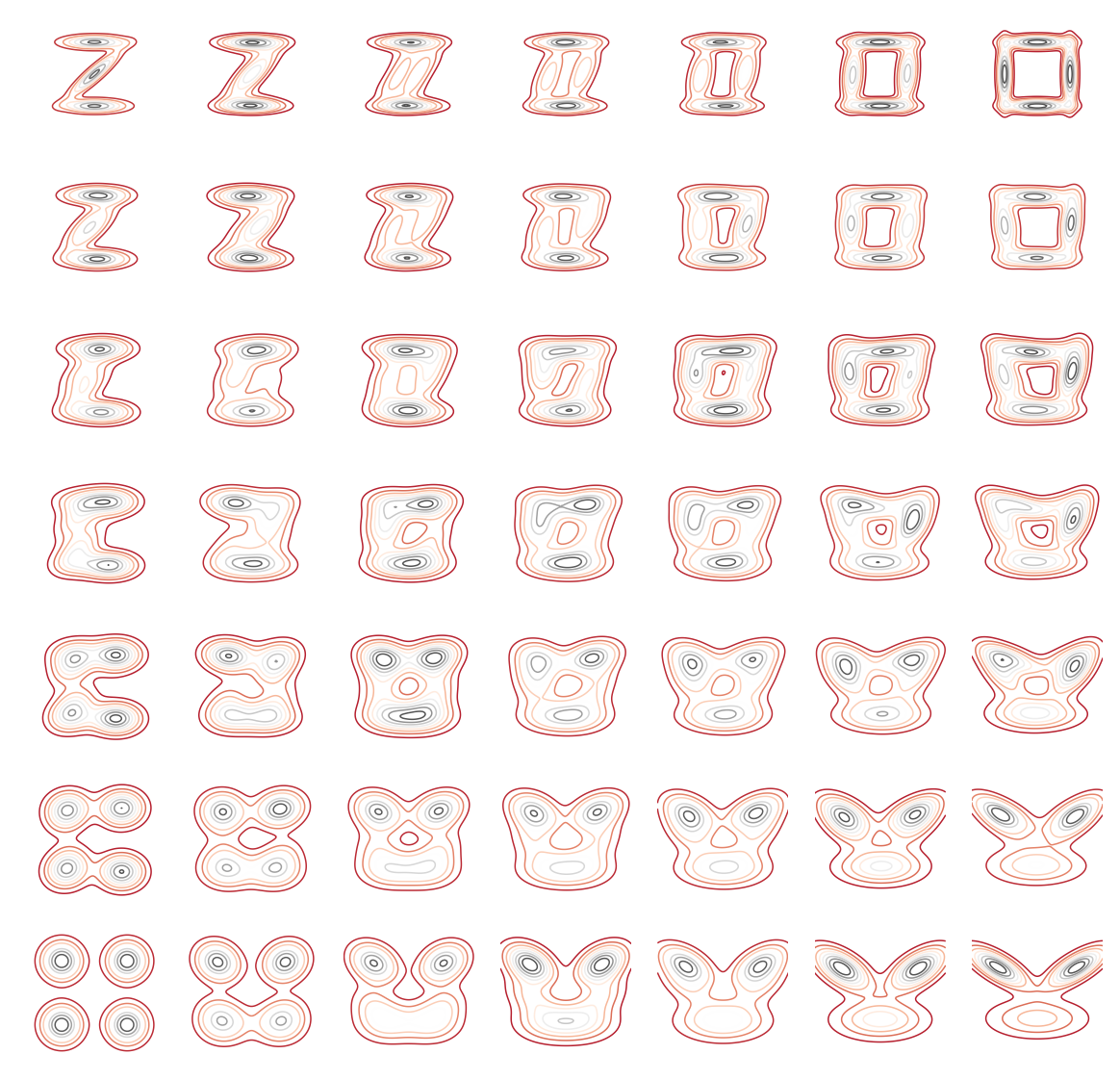}
  \includegraphics[width=6cm]{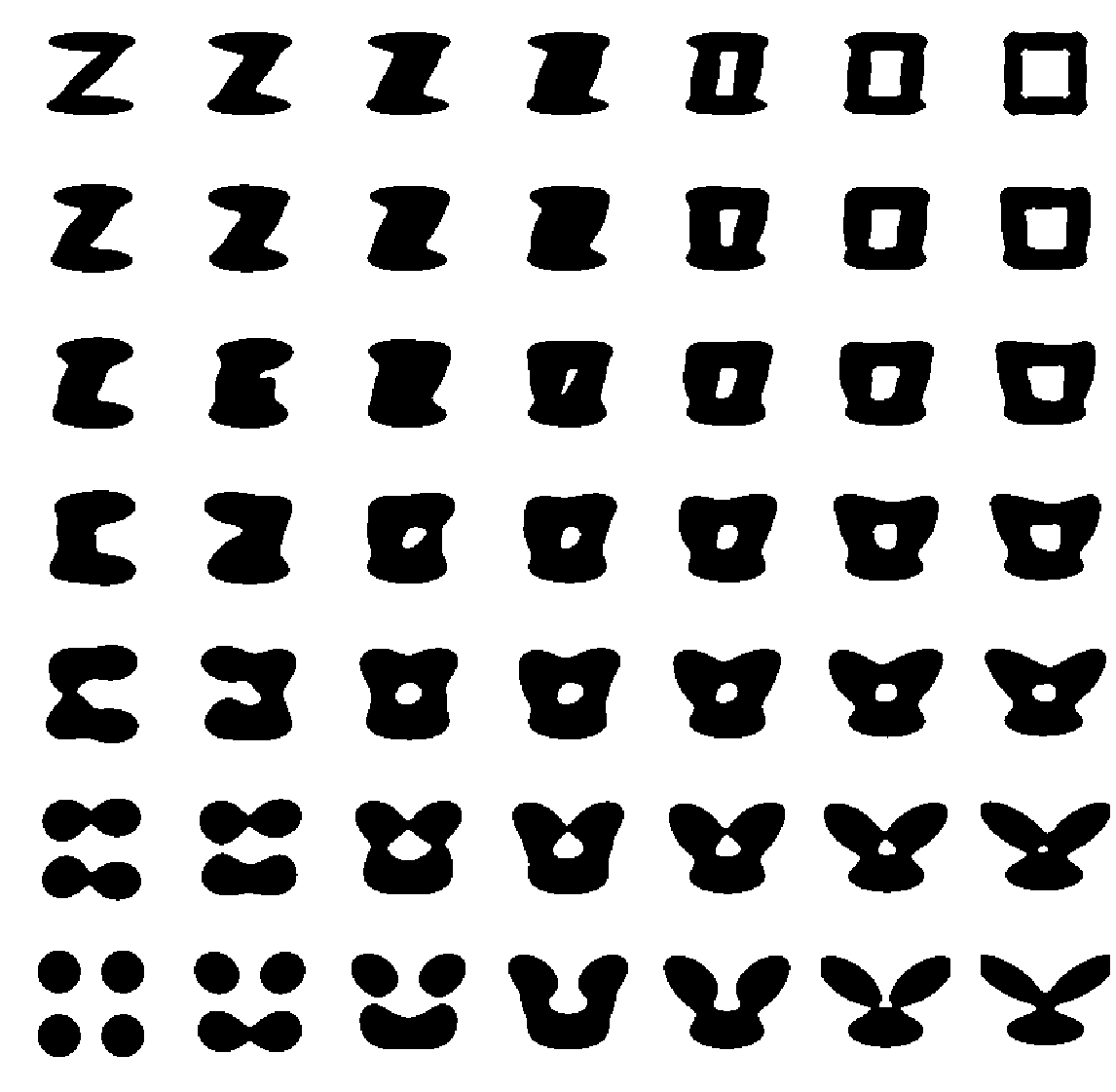}
  \caption{$MW_2$-barycenters between 4 Gaussian mixtures $\mu_0$, $\mu_1$, $\mu_2$ and $\mu_3$. On the left, some level sets of the distributions are displayed. On the right,  densities thresholded at level $1$ are displayed. We use bilinear weights with respect to the four corners of the square. 
\label{bary_4GMM}}
\end{figure}

To go further, Figure~\ref{bary_12GMM} shows barycenters where more
involved shapes have been approximated by mixtures of 12 Gaussian
components each. Observe that,  even if some of the original shapes (the star, the cross) have symmetries, these symmetries are not necessarily respected by the estimated GMM, and thus not preserved in the barycenters. This could be easily solved by imposing some symmetry in the GMM estimation for these shapes.   
\begin{figure}[h]
  \centering
  \includegraphics[width=12cm]{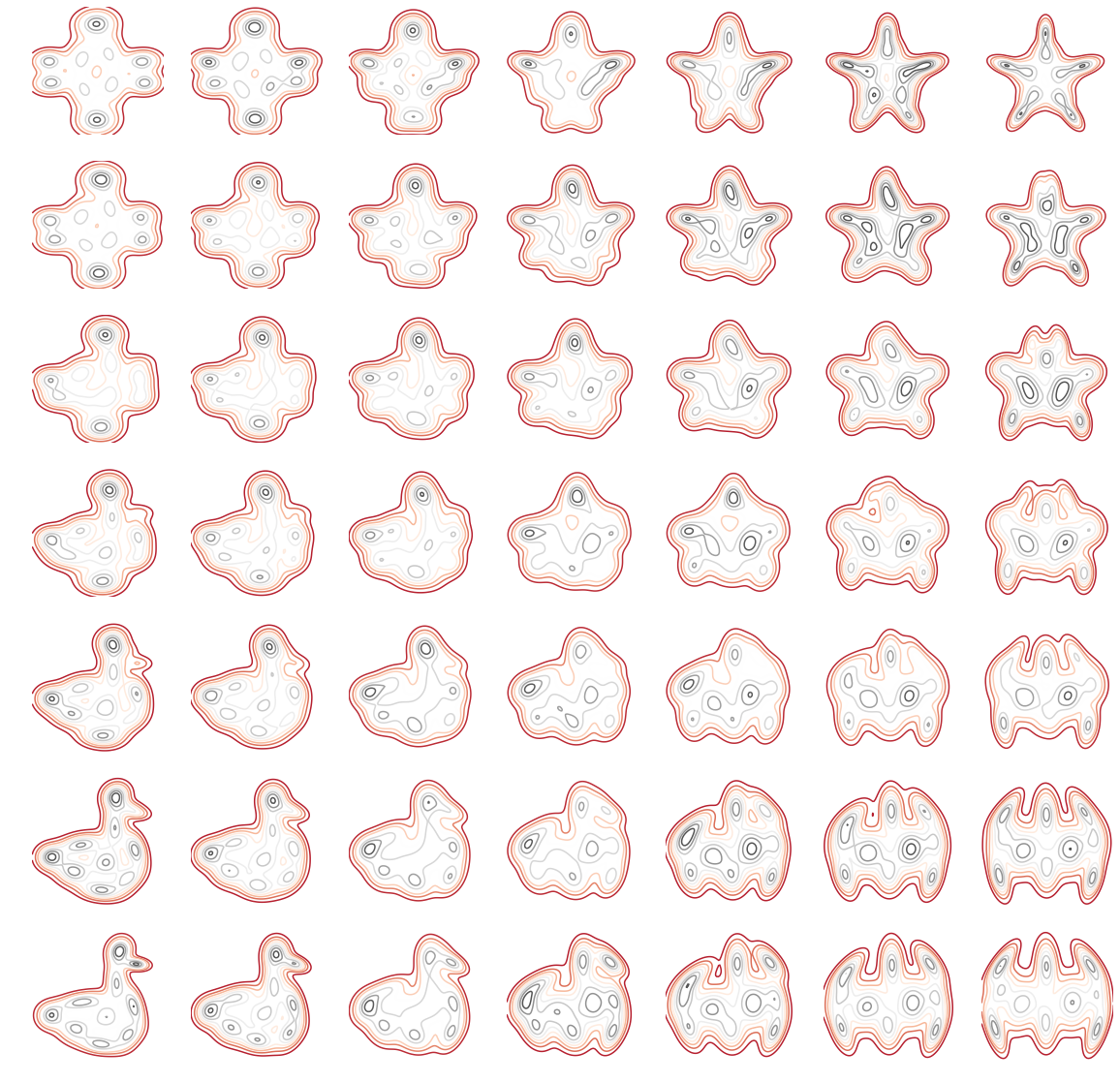}
  \caption{Barycenters between four mixtures of 12 Gaussian components, $\mu_0$, $\mu_1$, $\mu_2$, $\mu_3$ for the metric $MW_2$. The weights are bilinear with respect to the four corners of the square. 
\label{bary_12GMM}}
\end{figure}

\section{Using $MW_2$ in practice}
\label{sec:assignment}

\subsection{Extension to probability distributions that are not GMM}

Most applications of optimal transport involve data that do not follow a Gaussian mixture model and we can wonder how to make use of the distance $MW_2$ and the 
corresponding transport plans in this case. 
A simple solution is to approach these data by convenient Gaussian mixture models and to use the
transport plan $\gamma$ (or one of the maps defined in the previous
section) to displace the data.

Given two probability measures $\nu_0$ and $\nu_1$, we can define a pseudo-distance  $MW_{K,2}(\nu_0,\nu_1)$ as the distance
$MW_2({\mu}_0, {\mu}_1)$, where each ${\mu}_i$ ($i=0,1$)  is the Gaussian mixture model with $K$ components which minimizes an appropriate ``similarity measure'' to $\nu_i$. For instance, if   $\nu_i$ is a discrete measure $\nu_i = \frac 1 {J_i} \sum_{j=1}^{J_i} \delta_{x^i_j}$ in $\R^d$ , this similarity can be chosen as the opposite of the  log-likelihood of the
discrete set of points $\{x_j\}_{j=1,\dots, J_i}$ and the parameters of the Gaussian mixture can be infered thanks to the Expectation-Maximization
algorithm.  Observe that this log-likelihood can also be written
$$\E_{\nu_i }[\log {\mu}_i].$$
 
If $\nu_i$ is absolutely continuous, we can instead choose ${\mu}_i$ which minimizes $\KL(\nu_i,{\mu}_i )$ among GMM of order $K$. The discrete and continuous formulations coincide since  
$$\KL(\nu_i,{\mu}_i ) =  - H(\nu_i) - \E_{\nu_i }[\log {\mu}_i],$$
where $H(\nu_i)$ is the differential entropy of $\nu_i$. 

In both cases, the corresponding $MW_{K,2}$ does not define a distance since two different distributions may have the same corresponding Gaussian mixture. However, for $K$ large enough, their approximation by Gaussian mixtures will become different. The choice of $K$ must be a compromise between the quality of the approximation given by Gaussian mixture models and the affordable computing time. 
In any case, the optimal transport plan $\gamma_K$ involved in $MW_2({\mu}_0,{\mu}_1)$ can be used to compute an approximate transport map between $\nu_0$ and $\nu_1$.

In the experimental section, we will use this approximation for different data, generally with $K=10$.

\subsection{A similarity measure mixing $MW_2$ and $KL$}
\label{Extension:subsec}

In the previous paragraphs, we have seen how to use our Wasserstein-type distance $MW_2$ and its associated optimal transport plan on probability measures $\nu_0$ and $\nu_1$ that are not GMM.
Instead of a two step formulation (first an approximation by two GMM, and second the computation of $MW_2$), we propose here a relaxed formulation combining directly $MW_2$ with the Kullback-Leibler divergence.

Let $\nu_0$ and $\nu_1$ be two probability measures on $\R^d$, we
define
\begin{equation}
  \label{eq:EKl}
  E_{K,\lambda} (\nu_0,\nu_1)  = \min_{\gamma \in GMM_{2d}(K)}
 \int_{\R^d\times\R^d} \|y_0-y_1\|^2 d\gamma(y_0,y_1) -  \lambda\E_{\nu_0 }[\log P_0\#
  \gamma] -  \lambda\E_{\nu_1 }[\log P_1\#
  \gamma] ,
\end{equation}
where $\lambda>0$ is a parameter.

In the case where $\nu_0$ and $\nu_1$ are absolutely continuous with
respect to the Lebesgue measure, we can write instead
\begin{equation}
  \label{eq:EKltilde}
 \widetilde{E_{K,\lambda}}(\nu_0,\nu_1)  =\min_{\gamma \in GMM_{2d}(K)}
   \int_{\R^d\times\R^d} \|y_0-y_1\|^2 d\gamma(y_0,y_1) +
   \lambda \KL(\nu_0,P_0\#
  \gamma) +  \lambda \KL(\nu_1,P_1\#
  \gamma) 
\end{equation}
and $\widetilde{E_{K,\lambda}}(\nu_0,\nu_1) = E_{K,\lambda} (\nu_0,\nu_1) - \lambda H(\nu_0) -
\lambda H(\nu_1).$
Note that this formulation does not define a distance in
general.

This formulation is close to the unbalanced formulation of optimal
transport proposed by Chizat et al. in~\cite{Chizat2017ScalingAF}, with two differences:
a)  we constrain the solution $\gamma$ to be a GMM; and b) we use $\KL(\nu_0,P_0\#
  \gamma) $ instead of $\KL(P_0\#
  \gamma,\nu_0)$. In their case, the support of  $P_i\#
  \gamma$ must be contained in the support of $\nu_i$. When $\nu_i$ has a bounded support, this constraint is quite strong and would not make sense for a GMM $\gamma$.  

For discrete measures $\nu_0$ and $\nu_1$, when $\lambda$ goes to infinity, minimizing~\eqref{eq:EKl} becomes equivalent to approximate $\nu_0$
and $\nu_1$ by the EM algorithm and this only imposes the marginals of
$\gamma$ to be as close as possible to $\nu_0$ and $\nu_1$. When
$\lambda$ decreases, the first term favors solutions $\gamma$ whose
marginals become closer.  

Solving this problem (Equation \eqref{eq:EKl}) leads to computations similar to those
used in  the EM iterations~\cite{bishop2006pattern}. By differentiating with respect to the
weights, means and covariances of $\gamma$, we obtain equations which
are not in closed-form. For the sake of simplicity, we  illustrate
here what
happens in one dimension.  \\
Let $\gamma\in GMM_2(K)$ be a Gaussian mixture in dimension $2d=2$ with $K$
elements. We write 
$$\gamma = \sum_{k=1}^K \pi_k \mathcal{N}\left(
  \left( \begin{array}{c}  m_{0,k} \\ m_{1,k} \end{array} \right)
  , \left(\begin{array}{cc} 
     \sigma_{0,k}^2  & a_k \\ 
a_k & \sigma_{1,k}^2 \end{array} \right) \right) .$$
We have that the marginals are given by the 1d Gaussian mixtures
$$ P_0\# \gamma = \sum_{k=1}^K \pi_k \mathcal{N}(m_{0,k},
\sigma_{0,k}^2) \quad \text{ and } \quad  P_1\# \gamma = \sum_{k=1}^K \pi_k \mathcal{N}(m_{1,k}, \sigma_{1,k}^2) .$$

Then, to minimize, with respect to $\gamma$, the energy $E_{K,\lambda}(\nu_0,\nu_1)$ above, since
the $\KL$ terms are independent of the $a_k$, we can directly take
$a_k=\sigma_{0,k}\sigma_{1,k}$, and the transport cost term becomes
$$ \int_{\R^d\times\R^d} \|y_0-y_1\|^2 d\gamma(y_0,y_1)  =
\sum_{k=1}^K \pi_k \left[ (m_{0,k} - m_{1,k})^2 +  (\sigma_{0,k} -
  \sigma_{1,k})^2 \right] . $$
Therefore, we have to consider the problem of minimizing the following
``energy'':
\begin{eqnarray*}
F(\gamma) & = & \sum_{k=1}^K \pi_k \left[ (m_{0,k} - m_{1,k})^2 +  (\sigma_{0,k} -
  \sigma_{1,k})^2 \right]  \\
& & -  \lambda\int_\R \log\left( \sum_{k=1}^K \pi_k
  g_{m_{0,k}, \sigma_{0,k}^2}(x) \right)  d\nu_0(x) -  \lambda\int_\R \log\left( \sum_{k=1}^K \pi_k
  g_{m_{1,k}, \sigma_{1,k}^2}(x) \right)  d\nu_1(x) .
\end{eqnarray*}
It can be optimized through a simple gradient descent on the
parameters $\pi_k$, $m_{i,k}$, $\sigma_{i,k}$ for $i=0,1$ and
$k=1,\ldots, K$. Indeed a simple calculus shows that we can write
$$\frac{\partial F(\gamma)}{\partial \pi_k} = \left[ (m_{0,k} - m_{1,k})^2 +  (\sigma_{0,k} -
  \sigma_{1,k})^2 \right] - \lambda \frac{\tilde{\pi}_{0,k} +\tilde{\pi}_{1,k}
}{\pi_k} ,$$

 $$\frac{\partial F(\gamma)}{\partial m_{i,k}} =  2
 \pi_k(m_{i,k} - m_{i,k}) -  \lambda
 \frac{\tilde{\pi}_{i,k}}{\sigma_{i,k}^2}(\tilde{m}_{i,k} - m_{i,k} )
 ,$$ 

$$\text{ and } \quad \frac{\partial F(\gamma)}{\partial \sigma_{i,k}} = 2
 \pi_k(\sigma_{i,k} - \sigma_{j,k}) -  \lambda
 \frac{\tilde{\pi}_{i,k}}{\sigma_{i,k}^3}(\tilde{\sigma}_{i,k}^2 - \sigma_{i,k}^2 )
 ,$$  
where we have introduced some auxilary empirical estimates of the
variables given, for $i=0,1$ and $k=1,\ldots, K$, by
$$ \gamma_{i,k}(x) =  \frac{\pi_k g_{m_{i,k}, \sigma_{i,k}^2}(x)}{\sum_{l=1}^K \pi_l
  g_{m_{i,l}, \sigma_{i,l}^2}(x)} \quad \text{ and } \quad \tilde{\pi}_{i,k} = \int
\gamma_{i,k}(x) d\nu_i(x) ; $$
$$ \tilde{m}_{i,k} = \frac{1}{\tilde{\pi}_{i,k}} \int
 x \gamma_{i,k}(x) d\nu_i(x) \quad \text{ and } \quad \tilde{\sigma}_{i,k}^2 = \frac{1}{\tilde{\pi}_{i,k}} \int
 (x-m_{i,k})^2 \gamma_{i,k}(x) d\nu_i(x) .$$
Automatic differenciation of $F$ can also be used in practice. At each iteration of the gradient descent, we project on the
constraints $\pi_k\geq 0$, $\sigma_{i,k}\geq 0$ and $\sum_k \pi_k
=1$. 

On Figure \ref{fig:MW2KL}, we illustrate this approach on a simple
example.  The distributions $\nu_0$ and $\nu_1$ are 1d discrete distributions, plotted
as the red and blue histograms. On this example, we choose $K=3$ and
we use automatic differenciation (with the torch.autograd Python library) for
the sake of convenience. The red and blue plain curves represent the final distributions $P_0\#\gamma$ and $ P_1\#\gamma$,
for $\lambda$ in the set $\{10, 2, 1, 0.5, 0.1, 0.01\}$. The
behavior is as expected: when $\lambda$ is large, the KL terms are
dominating and the distribution $\gamma$ tends to have its marginal
fitting well the two distributions $\nu_0$ and $\nu_1$.
Whereas, when $\lambda$ is small, the Wasserstein transport term
dominates and the two marginals of $\gamma$ are almost equal.

\begin{figure}[htbp]
\centering
\subfloat[$\lambda=10$]{\includegraphics[width=5cm]{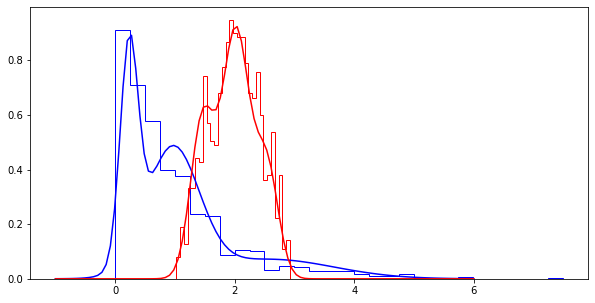}}\hfill
\subfloat[$\lambda=2$]{\includegraphics[width=5cm]{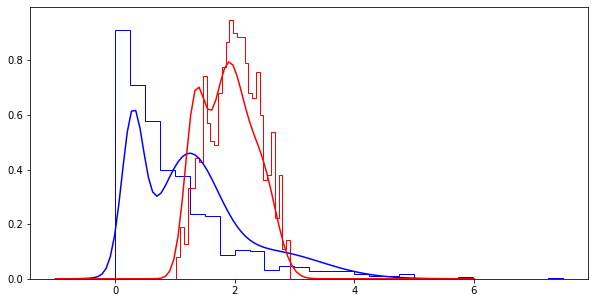}}\hfill
\subfloat[$\lambda=1$]{\includegraphics[width=5cm]{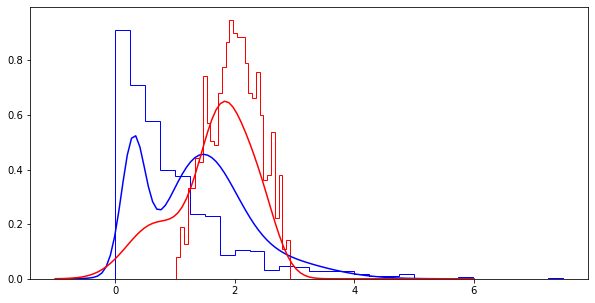}}

\subfloat[$\lambda=0.5$]{\includegraphics[width=5cm]{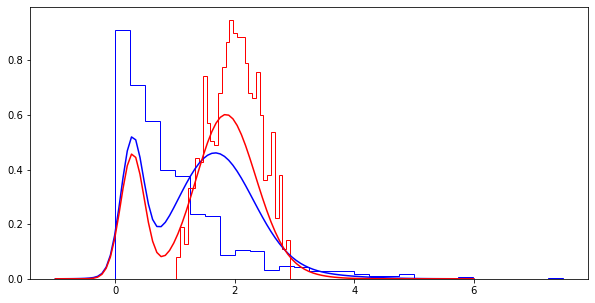}}\hfill
\subfloat[$\lambda=0.1$]{\includegraphics[width=5cm]{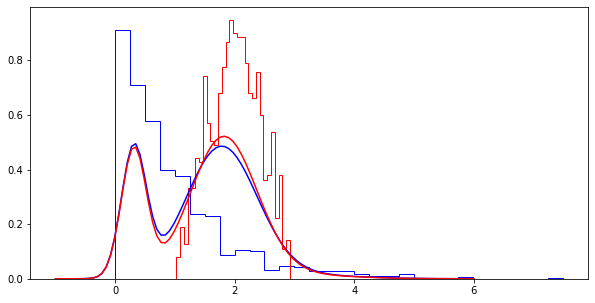}}\hfill
\subfloat[$\lambda=0.01$]{\includegraphics[width=5cm]{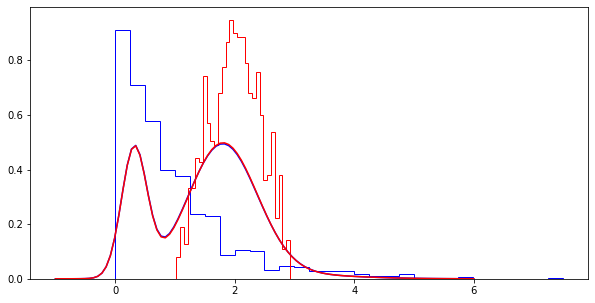}}

\caption{ The distributions $\nu_0$ and $\nu_1$ are 1d discrete distributions, plotted
as the red and blue discrete histograms. The
red and blue plain curves represent the final distributions
$P_0\#\gamma$ and $ P_1\#\gamma$. In this experiment, we use $K=3$ Gaussian components for $\gamma$.\label{fig:MW2KL} }
\end{figure}

\subsection{From a GMM transport plan to a transport map}
Usually, we need not only to have an optimal transport
plan and its corresponding cost, but also an assignment giving for each $x\in\R^d$ a 
corresponding value $T(x)\in\R^d$.
Let $\mu_0$ and $\mu_1$ be two GMM. Then, the optimal
transport plan between $\mu_0$ and $\mu_1$ for 
$MW_2$ is given by
$$  \gamma (x,y) = \sum_{k,l} w_{k,l}^\ast g_{m_0^k,\Sigma_0^k}(x)
\delta_{y=T_{k,l}(x)} .$$
It is not of the form $(\mathrm{Id},T)\#\mu_0$ (see also Figure \ref{1D_example_plans} for an
example), but we can however define a unique
assignment of each $x$, for
instance by setting 
$$ T_{mean}(x) = \E_\gamma (Y | X=x) ,$$
where here $(X,Y)$ is distributed according to the probability
distribution $\gamma$. Then, 
since the distribution of $Y|X=x$ is given by the discrete distribution 
$$ \sum_{k,l}  p_{k,l}(x)  \delta_{T_{k,l}(x)}  \quad \text{ with } \quad
p_{k,l}(x) = \frac{w_{k,l}^\ast g_{m_0^k,\Sigma_0^k}(x)}{\sum_{j} \pi_0^j g_{m_0^j,\Sigma_0^j}(x)} ,$$
we get that
$$ T_{mean}(x) = \frac{\sum_{k,l}
w_{k,l}^\ast g_{m_0^k,\Sigma_0^k}(x) T_{k,l}(x)}{\sum_{k} \pi^k_{0} g_{m_0^k,\Sigma_0^k}(x)} .$$

Notice that the $T_{mean}$ defined this way is an assignment that
will not necessarily satisfy the properties of an optimal transport
map. In particular, in dimension $d=1$, the map $T_{mean}$ may not be
increasing: each $T_{k,l}$ is increasing but because of the weights
that depend on $x$, their weighted sum is not necessarily
increasing. Another issue is that $T_{mean}\#\mu_0$ may be ``far''
from the target distribution $\mu_1$. This happens for instance, in 1D, when
$\mu_0=\mathcal{N}(0,1)$ and $\mu_1$ is the mixture of
$\mathcal{N}(-a,1)$ and $\mathcal{N}(a,1)$, each with weight $0.5$. In
this extreme case we even have that $T_{mean}$ is the identity map,
and thus $T_{mean}\#\mu_0=\mu_0$, that can be very far from $\mu_1$ when $a$ is
large.

Now, another way to define an assignment is to define it as a random
assignment using the optimal plan $\gamma$. More precisely, for a
fixed value $x$ we can
define
$$T_{rand}(x) =  T_{k,l}(x) \quad \text{ with probability } p_{k,l}(x)
= \frac{w_{k,l}^\ast g_{m_0^k,\Sigma_0^k}(x)}{\sum_{j} \pi_0^j
  g_{m_0^j,\Sigma_0^j}(x)} .$$
Observe that, from a mathematical point of view, we can define a random variable $T_{rand}(x)$ for a fixed
value of $x$, or also a finite set of independent random variables
$T_{rand}(x)$  for a finite set of $x$. But constructing and defining  $T_{rand}$ as a
stochastic process on the whole space $\R^d$ would be mathematically
much more difficult (see \cite{Kallenberg} for instance). 

Now, for any measurable set $A$ of $\R^d$ and any $x\in\R^d$, we can 
define the map $\kappa(x,A) := \mathbb{P}[T_{rand}(x) \in A],$ and we have
$$\kappa(x,A) = \frac{\gamma(x,A)}{\sum_{j} \pi_0^j
  g_{m_0^j,\Sigma_0^j}(x) },  \;
\text{ and thus }\;\;\;
\int\kappa(x,A) d\mu_0(x) = \mu_1(A).$$
It means that if the measure $T_{rand}$ could be defined everywhere,
then ``$T_{rand} \# \mu_0$'', would be equal in
expectation to $\mu_1$.

Figure~\ref{fig:points_displacement} illustrates these two possible
assignments $T_{mean}$ and $T_{rand}$ on a simple example.  In this example, two discrete
measures $\nu_0$ and $\nu_1$ are approximated by Gaussian mixtures
$\mu_0$ and $\mu_1$ of order $K$, and we compute the transport maps
$T_{mean}$ and $T_{rand}$ for these two mixtures.  These maps are used
to displace the points of $\nu_0$. We show the result of these
displacements for different values of $K$. We can see that depending
on the configuration of points, the results provided by $T_{mean}$ and
$T_{rand}$ can be quite different. As expected, the measure $T_{rand} \# \nu_0$
(well-defined since $\nu_0$ is composed of a finite set of points)
looks more similar to $\nu_1$ than  $T_{mean} \# \nu_0$ does. And
$T_{rand}$ is also less regular than  $T_{mean}$ (two close points can be easily displaced to two positions far from each other). This may not be desirable in some applications, for instance in color transfer as we will see in Figure \ref{fig:renoir_gauguin_histograms} in the experimental section.

\begin{figure}[h!]
  \centering \includegraphics[width=7cm, angle = 90]{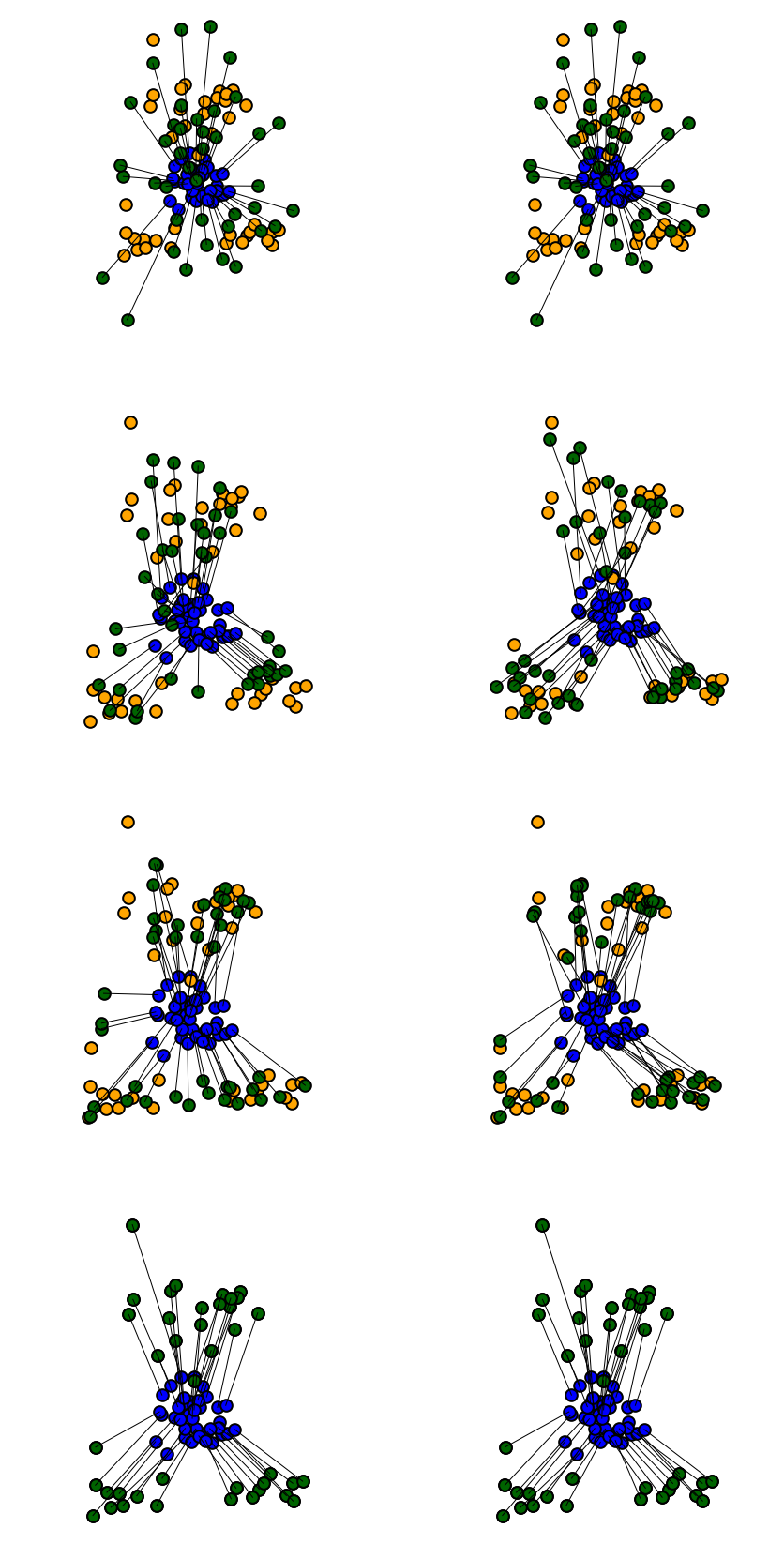}
    \caption{Assignments between two point clouds $\nu_0$ (in blue) and $\nu_1$ (in yellow) composed of $40$ points, for different values of $K$. Green points represent $T\#\nu_0$, where $T=T_{rand}$ on the first line and $T=T_{mean}$ on the second line. The four columns correspond respectively to $K=1,5,10,40$.  Observe that for $K=1$, only one Gaussian is used for each set of points, and $T\#\nu_0$ is quite far from $\nu_1$ (in this case, $T_{rand}$ and $T_{mean}$ coincide). When $K$ increases, the discrete distribution $T\#\nu_0$ becomes closer to $\nu_1$, especially for $T=T_{rand}$. When $K$ is chosen equal to the number of points, we obtain the result of the $W_2$-optimal transport between $\nu_0$ and $\nu_1$.  
\label{fig:points_displacement} }
\end{figure}

\section{Two applications in image processing}
\label{sec:applications}
  
We have already illustrated the behaviour of the distance $MW_2$ in small dimension. In the following, we investigate more involved examples in larger dimension. 
 In the last ten years, optimal transport has been thoroughly used for
 various applications in image processing and computer vision,
 including color transfer, texture synthesis, shape matching. We focus
 here on two simple applications: on the one hand, color transfer,
 that involves to transport mass in dimension $d=3$ since color
 histograms are 3D histograms, and on the other hand patch-based
 texture synthesis, that necessitates transport in  dimension $p^2$
 for $p\times p$ patches. These two applications require to compute
 transport plans or barycenters between potentially millions of
 points. We will see that the use of $MW_2$ makes these computations
 much easier and faster than the use of classical optimal transport,
 while yielding excellent visual results.   
The codes of the different experiments are available through
Jupyter notebooks on \url{https://github.com/judelo/gmmot}.

\subsection{Color transfer}
\label{sec:colortransfer}
We start with the problem of color transfer. A discrete color image
can be seen as a function $u:\Omega \rightarrow \R^3$ where $\Omega =
\{0,\dots n_r-1\}\times\{0,\dots n_c-1\}$ is a discrete grid. The
image size is $n_r\times n_c$ and for each $i \in \Omega$, $u(i)\in
\R^3$ is a set of three values corresponding to the intensities of
red, green and blue in the color of the pixel. Given two images $u_0$
and $u_1$ on grids $\Omega_0$ and $\Omega_1$, we define the discrete
color distributions $\eta_k = \frac 1{|\Omega_k|} \sum_{i\in\Omega_k}
\delta_{u_k(i)}$, $k=0,1$, and we approximate these two distributions
by Gaussian mixtures $\mu_0$ and $\mu_1$ thanks to the
Expectation-Maximization (EM) algorithm\footnote{In practice, we use
  the \textit{scikit-learn} implementation of EM with the
  \textit{kmeans} initialization.}.  Keeping the notations used previously
in the paper, we write $K_k$ the number of Gaussian components in the mixture
$\mu_k$, for $k=0,1$. We compute the $MW_2$ map between these two
mixtures and the corresponding $T_{mean}$. We use it to compute
$T_{mean}(u_0)$, an image with the same content as $u_0$ but with
colors much closer to those of $u_1$. Figure~\ref{fig:renoir_gauguin}
illustrates this process on two paintings by Renoir and Gauguin,
respectively \textit{Le déjeuner des canotiers} and \textit{Manhana no
  atua}. For this experiment, we choose $K_0=K_1=10$. 
The corresponding transport map for $MW_2$ is relatively fast to compute (less than one minute with a non-optimized Python implementation, using the POT library \cite{flamary2017pot} for computing the map between the discrete distributions of $10$ masses). We also show on
the same figure $T_{rand}(u_0)$, the result of the sliced optimal
transport~\cite{rabin2011wasserstein,bonneel2015sliced}, and the
result of the separable optimal transport (i.e. on each color channel separately). Notice that the
complete optimal transport on such huge discrete distributions
(approximately 800000 Dirac masses for these $1024\times 768$ images)
is hardly tractable in practice. As could be expected, the image
$T_{rand}(u_0)$ is much noisier than the image $T_{mean}(u_0)$. 
We show on
Figure~\ref{fig:renoir_gauguin_histograms} the discrete color
distributions of these different images and the corresponding classes provided by EM (each point is assigned to its most likely class).

\begin{figure}[h!]
  \centering \includegraphics[width=6cm]{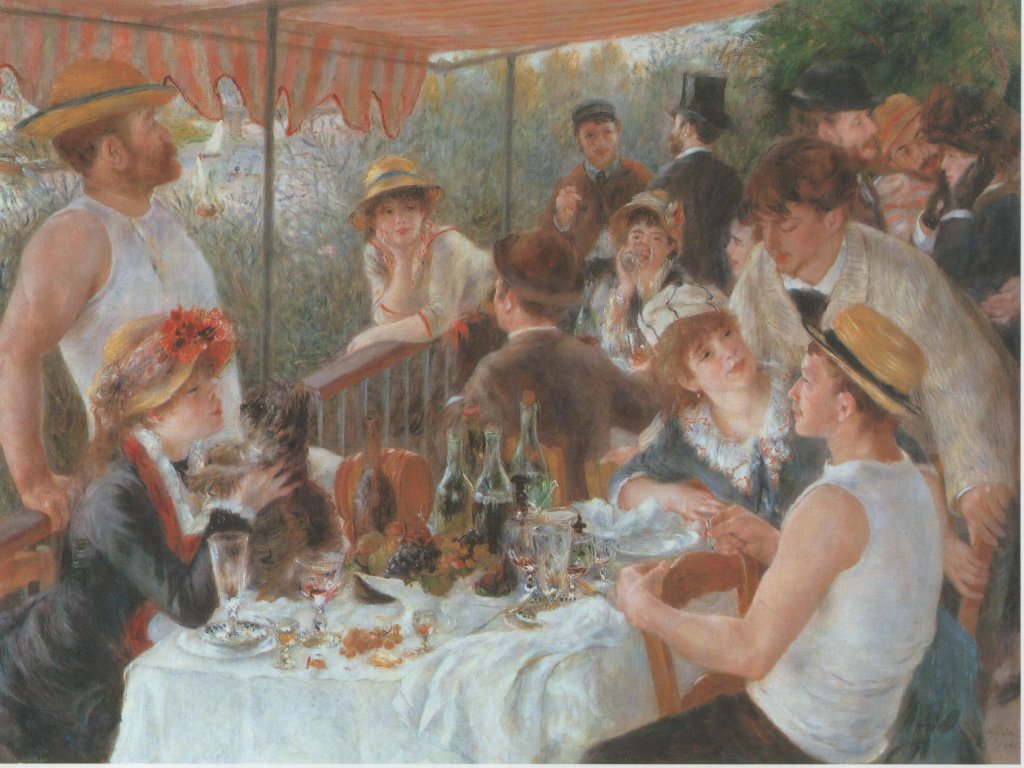}
\includegraphics[width=6cm]{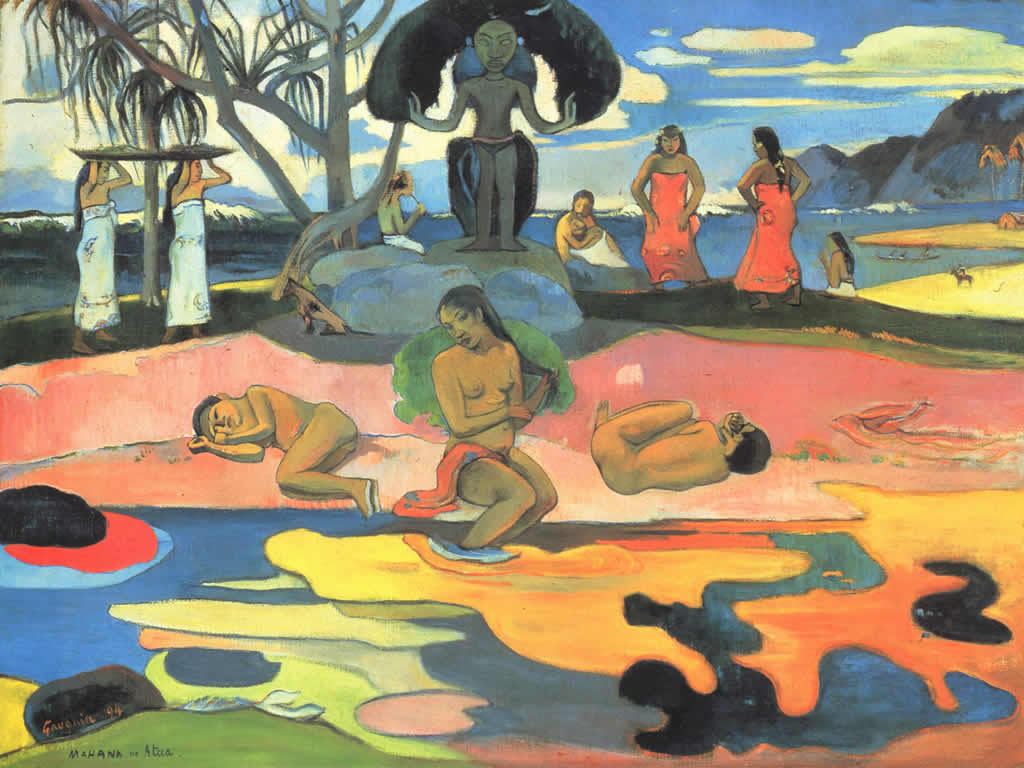}
\vspace{.3em}
\includegraphics[width=6cm]{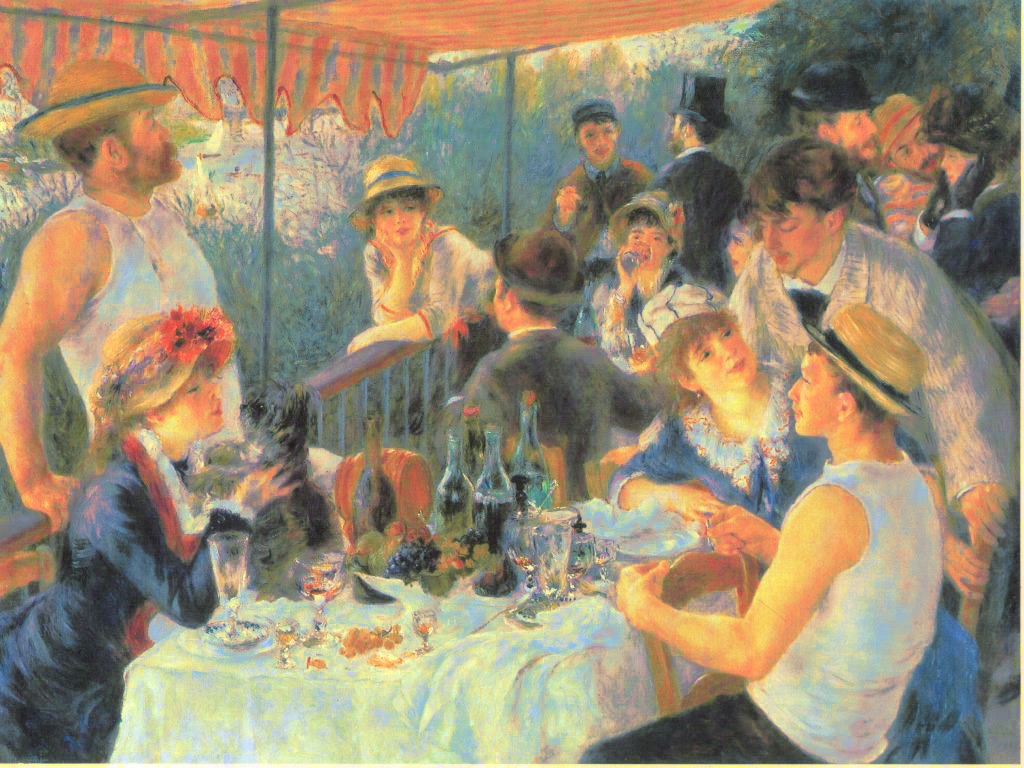}
\includegraphics[width=6cm]{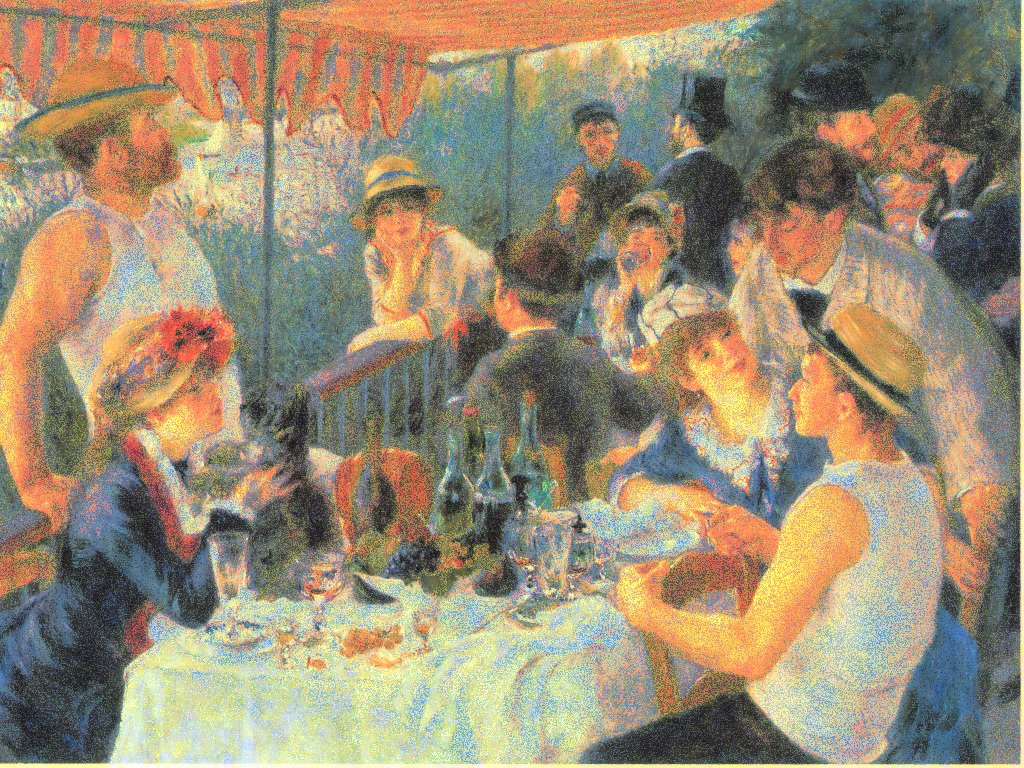}
\vspace{.3em}

\includegraphics[width=6cm]{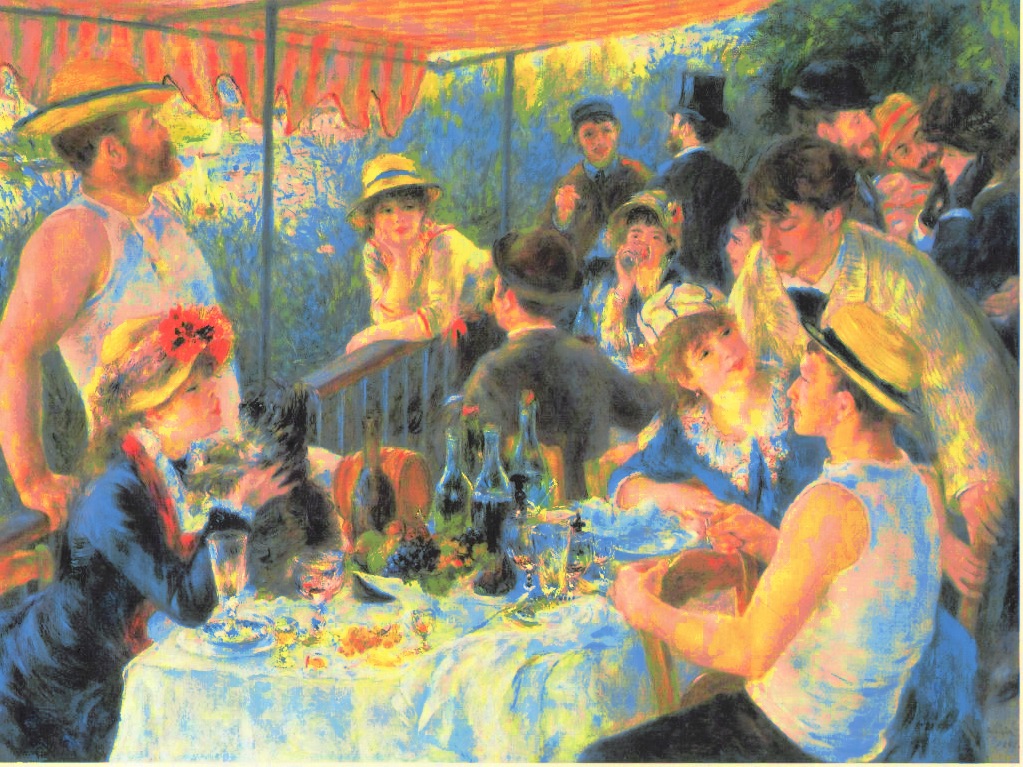}
\includegraphics[width=6cm]{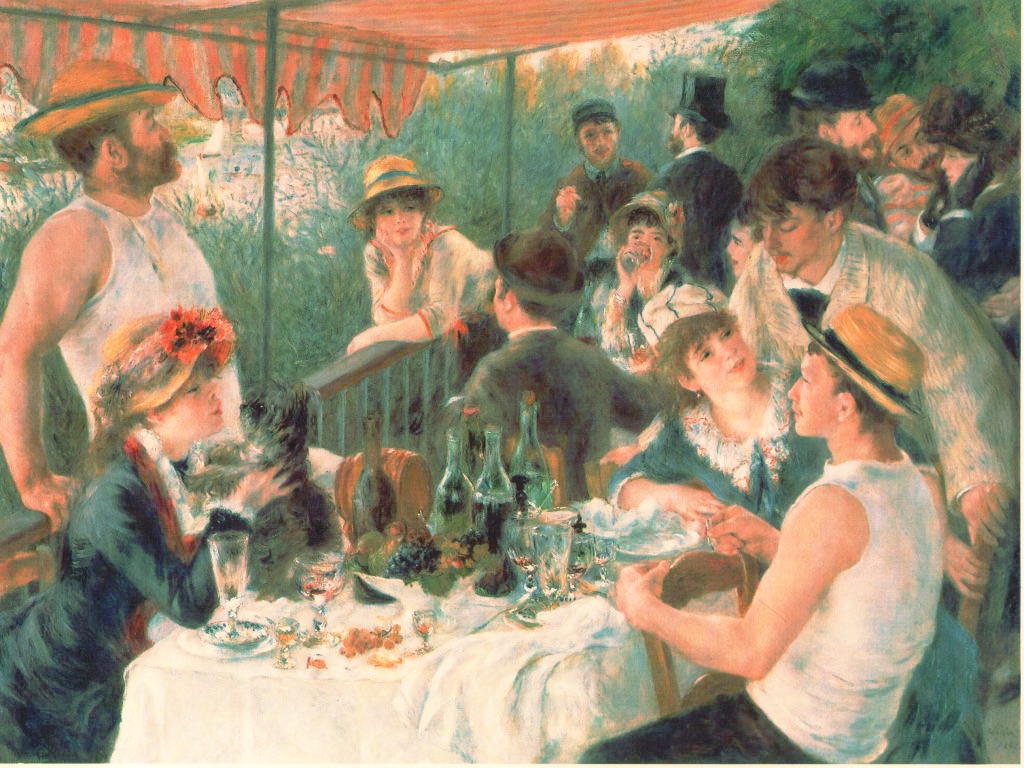}
\vspace{.3em}

    \caption{First line, images $u_0$ and $u_1$ (two paintings by
      Renoir and Gauguin). Second line,  $T_{mean}(u_0)$ and $T_{rand}(u_0)$. Third line, color
      transfer with the sliced optimal transport~\cite{rabin2011wasserstein,bonneel2015sliced}, that we denote by
      $SOT(u_0)$  and result of the separable optimal transport (color
      transfer is applied separately on the three dimensions -
      channels - of the color  distributions). 
\label{fig:renoir_gauguin} }
\end{figure}

\begin{figure}[h]
\begin{center}
\begin{tabular}{lll}
\includegraphics[width=4.5cm]{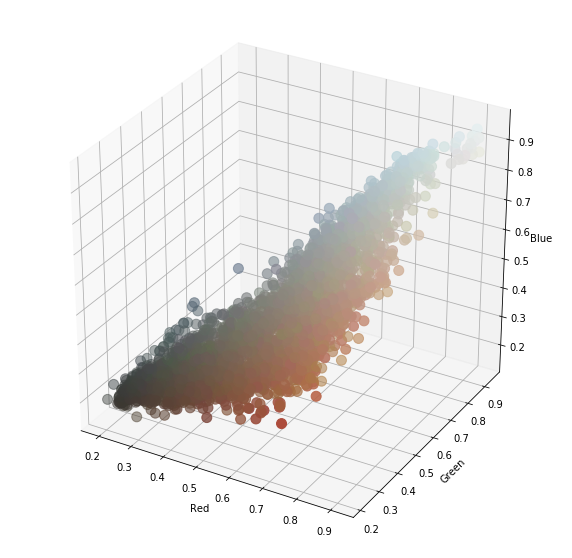}
  & \hspace{-1cm}
\includegraphics[width=4.5cm]{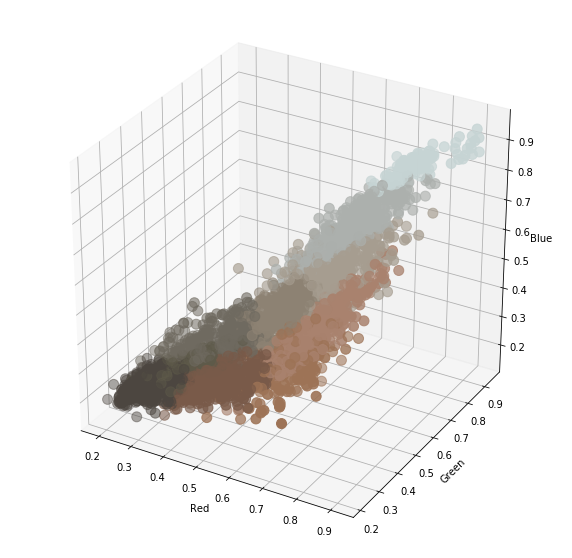}
 &\hspace{-1cm}
\includegraphics[width=4.5cm]{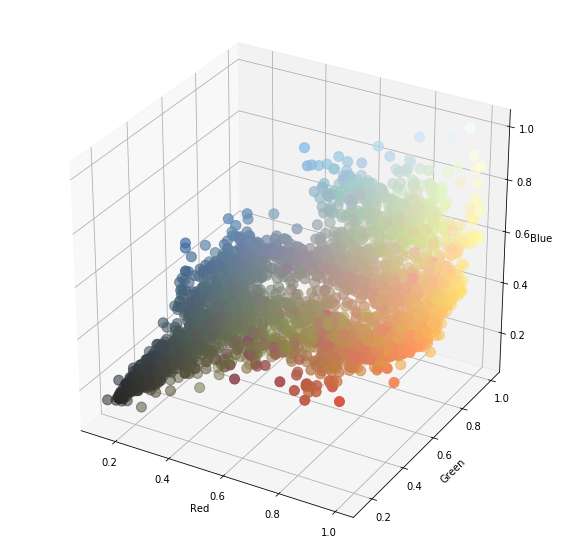} 
\\
\includegraphics[width=4.5cm]{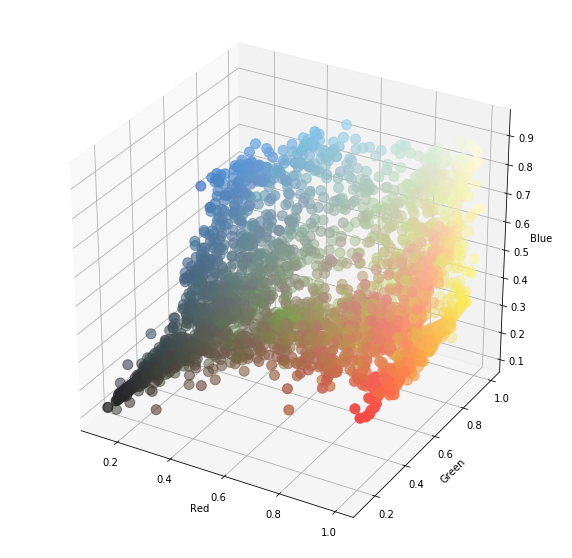}
&\hspace{-1cm}
\includegraphics[width=4.5cm]{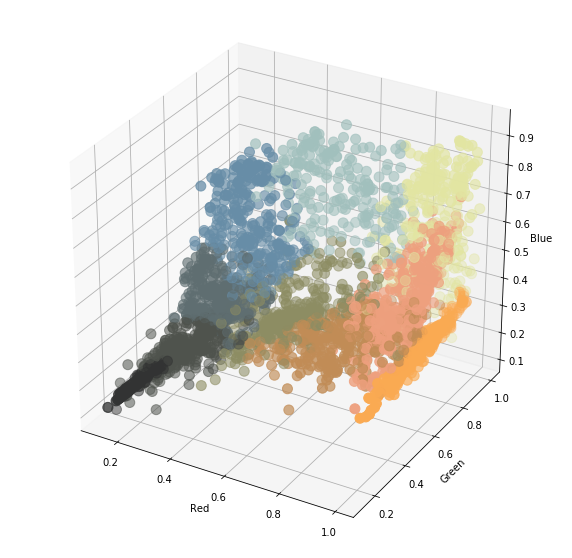}
 & \hspace{-1cm}
\includegraphics[width=4.5cm]{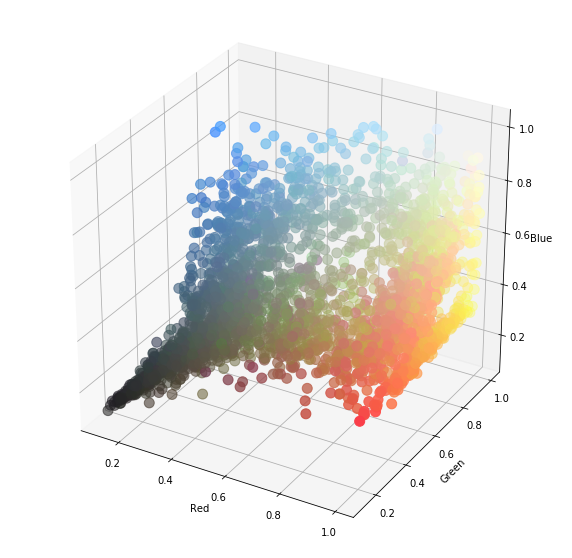}
\end{tabular}
\end{center}
  \caption{The images $u_0$ and $u_1$ are the ones of
    Figure~\ref{fig:renoir_gauguin}.  First line: color distribution of the image $u_0$, the
    $10$ classes found by the EM algorithm, and color distribution of
    $T_{mean}(u_0)$. Second line: color distribution of the image $u_1$, the
    $10$ classes found by the EM algorithm, and color distribution of
    $T_{rand}(u_0)$.
\label{fig:renoir_gauguin_histograms}}
\end{figure}

The value $K=10$ that we have chosen here is the result of a
compromise. Indeed, when $K$ is too small, the approximation by the
mixtures is generally too rough to represent the complexity of the color data
properly.  At the opposite, we have observed that increasing the number
of components does not necessarily help since the corresponding
transport map will loose regularity. For color transfer experiments,
we found in practice that using around $10$ components yields the best
results.    We also illustrate this on Figure
\ref{fig:red_white_mountains}, where we show the results of the color
transfer with $MW_2$ for different values of $K$. On the different
images, one can appreciate how the color distribution gets closer and
closer to the one of the target image as $K$ increases. 

\begin{figure}[h]
\begin{center}
\begin{tabular}{ccccc}
\includegraphics[width=3cm]{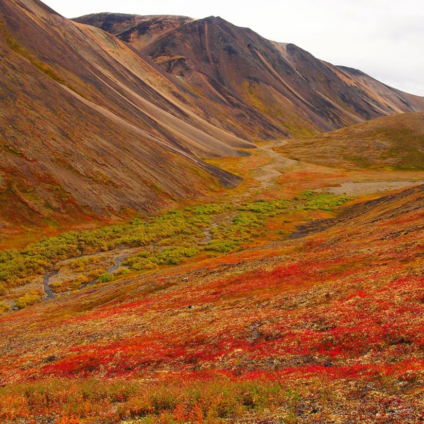}
  &  \hspace{-0.5cm}
\includegraphics[width=3cm]{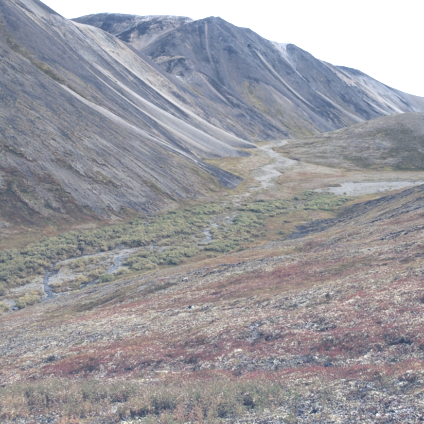}
 & \hspace{-0.5cm}
\includegraphics[width=3cm]{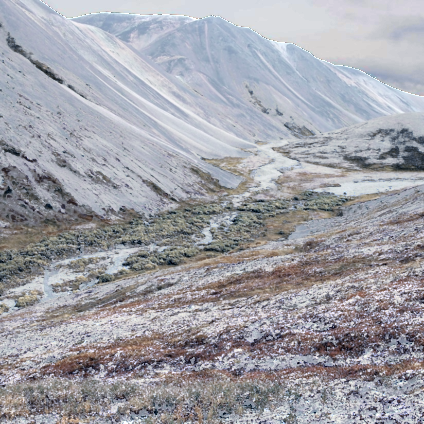} 
& \hspace{-0.5cm}
\includegraphics[width=3cm]{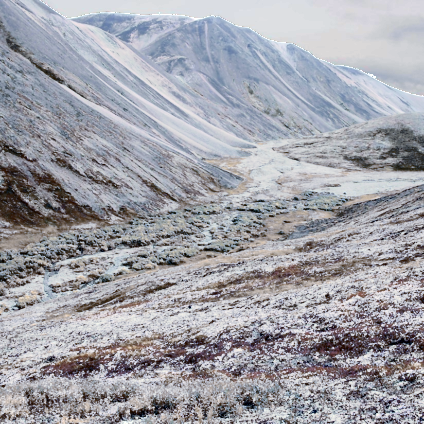}
& \hspace{-0.5cm}
\includegraphics[width=3cm]{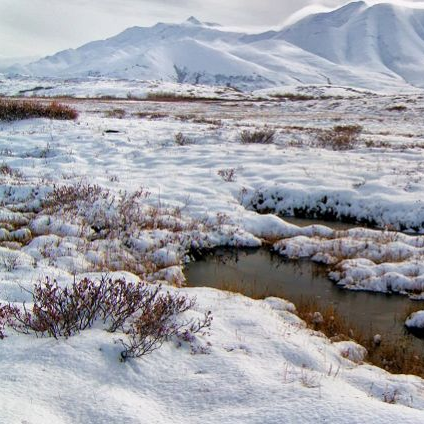}
  \\
 & $K=1$ & $K=3$ &  $K=10$ & 
\end{tabular}
\end{center}
  \caption{The left-most image is the ``red mountain'' image, and its
    color distribution is modified to match the one of the right-most image (the
    ``white mountain'' image) with $MW_2$ using respectively $K=1$,
    $K=3$ and $K=10$ components in the Gaussian mixtures. 
\label{fig:red_white_mountains}}
\end{figure}

We end this section with a color manipulation experiment, shown on
Figure~\ref{fig:cat_barycenter}. Four different images being given, we
create barycenters for $MW_2$ between their four color palettes
(represented again by mixtures of 10 Gaussian components), and we modify the first of the four images so that its color palette spans this space of barycenters. For this experiment (and this experiment only), a spatial regularization step is applied in post-processing~\cite{rabin2011removing} to remove some artifacts created by these color transformations between highly different images.

\begin{figure}[h!]
  \centering \includegraphics[width=13cm]{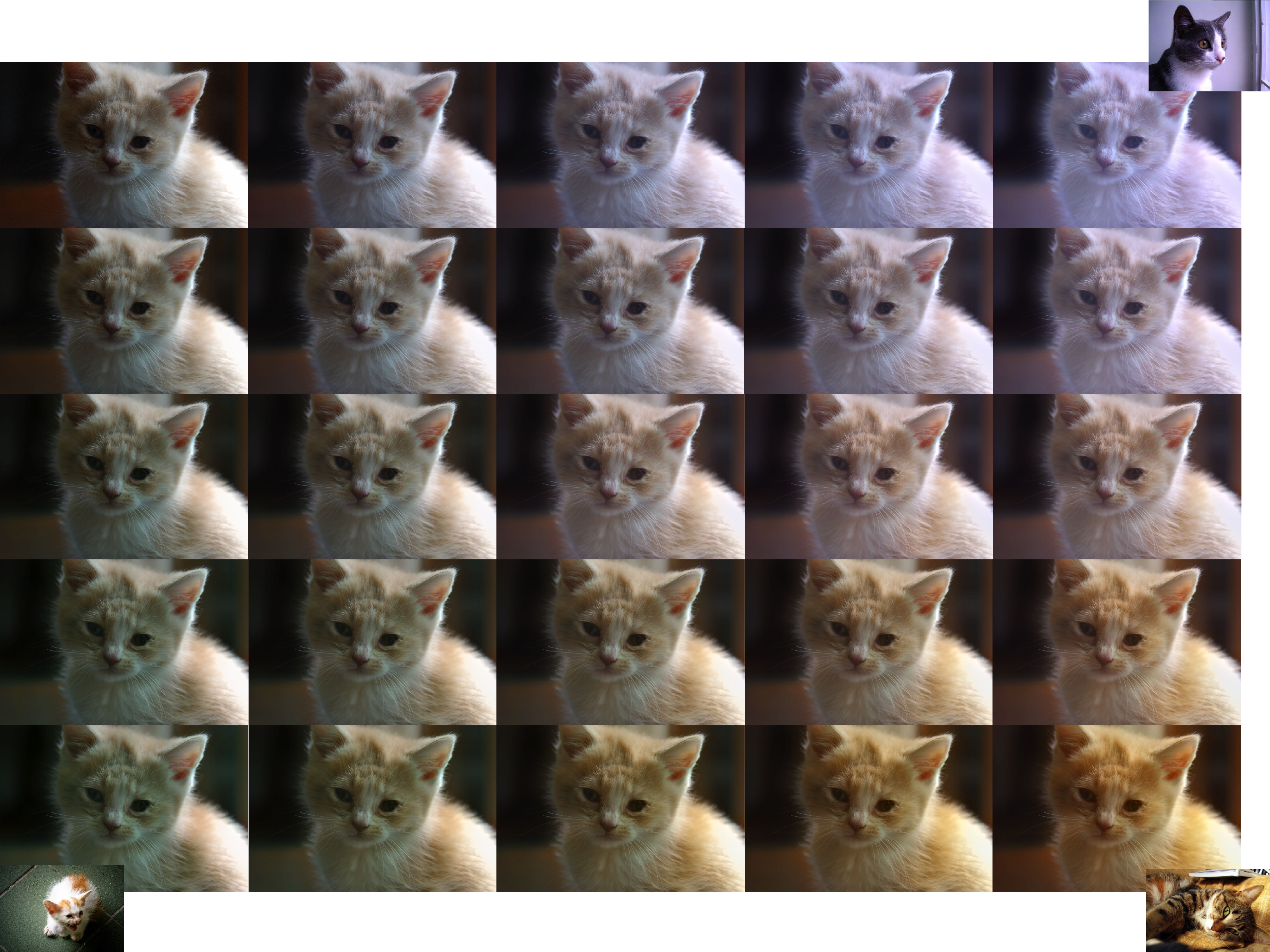}
    \caption{In this experiment, the top left image is modified in
      such a way that its color palette goes through the
      $MW_2$-barycenters between the color palettes of the four corner
      images. Each color palette is represented as a mixture of 10
      Gaussian components. The weights used for the barycenters are
      bilinear with respect to the four corners of the rectangle.  
\label{fig:cat_barycenter} }
\end{figure}

\subsection{Texture synthesis}

Given an exemplar texture image $u:\Omega \rightarrow R^3$, the goal
of texture synthesis is to synthetize images with the same perceptual
characteristics as $u$, while keeping some innovative content. The
literature on texture synthesis is rich, and we will only focus here
on a  bilevel approach proposed recently
in~\cite{galerne2017semi}. The method relies on the optimal transport
between a continuous (Gaussian or Gaussian mixtures) distribution and
a discrete distribution (distribution of the patches of the exemplar
texture image). The first step of the method can be described as follows. For a given exemplar image $u:\Omega \rightarrow R^3$, the authors compute the asymptotic discrete spot noise (ADSN) associated with $u$, which is the stationary Gaussian random field $U:\mathbb{Z}^2 \rightarrow \R^3$ with same mean and covariance as $u$, {\em i.e.}
\[\forall x \in \mathbb{Z}^2, \; U(x) = \bar{u}+ \sum_{y\in
  \mathbb{Z}^2} t_u(y)W(x-y), \; \text{ where }
  \begin{cases}
    \bar{u} = \frac 1 {|\Omega|} \sum_{x\in \Omega} u(x)\\
t_u =  \frac 1 {\sqrt{|\Omega|}} (u-\bar{u}) \mathbf{1}_{\Omega}, 
  \end{cases}
\]
with $W$ a standard normal Gaussian white noise on $\mathbb{Z}^2$.
Once the ADSN $U$ is computed, they extract a set $S$ of 
$p\times p$ sub-images (also called \textit{patches}) of $u$. In our
experiments, we extract one
patch for each pixel of $u$ (excluding the borders), so patches are
overlapping and the number of patches is approximately equal to the
image size. The authors of~\cite{galerne2017semi} then  
define $\eta_1$ the empirical distribution of this set of patches
(thus $\eta_1$ is in dimension $3\times p \times p$, {\em i.e.} $27$
for $p=3$) and $\eta_0$ the Gaussian distribution of patches of $U$,
and compute the semi-discrete optimal transport map $T_{SD}$ from
$\eta_0$ to $\eta_1$. This map $T_{SD}$ is then applied to each patch
of a realization of $U$, and an ouput synthetized image $v$ is
obtained by averaging the transported patches at each pixel. Since the
semi-discrete optimal transport step is numerically very expensive in
such high dimension, we propose to make use of the $MW_2$ distance
instead. For that, we approximate the two discrete patch distributions
of $u$ and $U$ by Gaussian Mixture models $\mu_0$ and $\mu_1$, and we
compute the optimal map $T_{mean}$ for $MW_2$ between them. The rest
of the algorithm is similar to the one described
in~\cite{galerne2017semi}. Figure~\ref{fig:texture} shows the results for different choices of  exemplar images $u$.
In practice, we use $K_0=K_1=10$, as in
color transfer, and $3\times 3$ color patches. The results obtained
with our approach are visually very similar to the ones obtained
with~\cite{galerne2017semi}, for a computational time approximately
10 times smaller. More precisely, for instance for an image of size
$256\times 256$, the proposed approach takes about $35$ seconds,
whereas the semi-discrete approach of \cite{galerne2017semi} takes about $400$ seconds.
We are currently exploring a multiscale version of
this approach, inspired by the recent~\cite{leclaire2019multi}. 

\begin{figure}[h]
\begin{center}
\includegraphics[width=4cm]{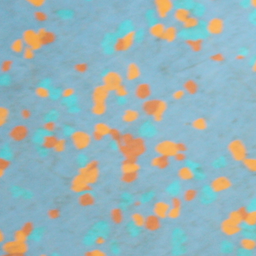}
\includegraphics[width=4cm]{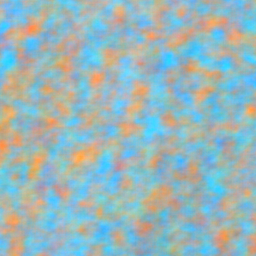}
\includegraphics[width=4cm]{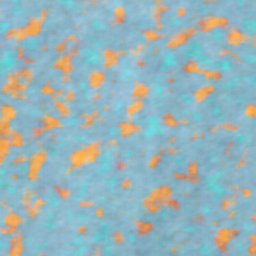}

\includegraphics[width=4cm]{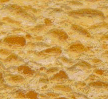}
\includegraphics[width=4cm]{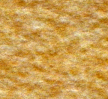}
\includegraphics[width=4cm]{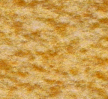}

\includegraphics[width=4cm]{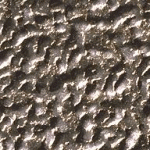}
\includegraphics[width=4cm]{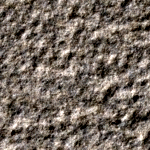}
\includegraphics[width=4cm]{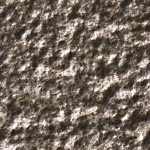}
\end{center}
\caption{Left, original texture $u$. Middle, ADSN $U$. Right, synthetized version.\label{fig:texture}}
\end{figure}

\section{Discussion and conclusion}

In this paper, we have defined a Wasserstein-type distance on the set
of Gaussian mixture models, by restricting the set of possible
coupling measures to Gaussian mixtures. We have shown that this
distance, with an explicit discrete formulation, is easy to compute
and suitable to compute transport plans or barycenters in high
dimensional problems where the classical Wasserstein distance remains
difficult to handle. We have also discussed the fact that the distance $MW_2$ could be
extended to other types of mixtures, as soon as we have a marginal
consistency property and an
identifiability property similar to the one used in the proof of
Proposition~\ref{prop:MW2}. In practice, Gaussian mixture models are
versatile enough to represent large classes of concrete and applied
problems. One important question raised by the introduced framework
and its generalization in Section \ref{Extension:subsec} is
how to estimate the mixtures for discrete data, since the obtained result
 will depend on the number $K$ of Gaussian components in the mixtures and on
the parameter $\lambda$ that weights the data-fidelity terms.
 If the number of Gaussian
components is chosen
large enough, and covariances small enough, the transport plan for
$MW_2$ will look very similar to the one of $W_2$, but at the price of a
high computational cost. If, on the contrary, we choose a very small
number of components (like in the color transfer experiments of
Section~\ref{sec:colortransfer}), the resulting optimal transport map
will be much simpler, which seems to be desirable for some applications.

\section*{Acknowledgments} 

We would like to thank Arthur Leclaire for his valuable assistance for
the texture synthesis experiments.

\section*{Appendix: proofs}

\subsection*{Density of $GMM_d(\infty)$ in $\mathcal{P}_p(\R^d)$}

\begin{lemmastar}[Lemma 3.1]
The set 
\[ \left\{ \sum_{k=1}^N \pi_k \delta_{y_k}\;;\; N\in\N,\;(y_k)_k \in (\R^{d})^N, \; (\pi_k)_k\in \Gamma_N \right\}\]
is dense in $\mathcal{P}_p(\R^d)$ for the metric $W_p$, for any $p\geq 1$. 
\end{lemmastar}

\begin{proof}
  The proof is adapted from the proof of Theorem 6.18 in~\cite{villani2008optimal} and given here for the sake of completeness.

Let $\mu \in \mathcal{P}_p(\R^d)$. For each $\epsilon>0$, we can find $r$ such that
$\int_{B(0,r)^c}\|y\|^pd\mu(x) \leq\epsilon^p$, where $B(0,r) \subset
\R^d$ is the ball of center $0$ and radius $r$, and $B(0,r)^c$ denotes
its complementary set in $\R^d$. The ball $B(0,r)$ can be covered by a
finite number of balls $B(y_k,\epsilon)$, $1\leq k \leq N$. Now,
define $B_k = B(y_k,\epsilon) \setminus \cup_{1\leq j<k}
B(y_j,\epsilon)$, all these sets are disjoint and still cover
$B(0,r)$. \\
Define $\phi:\R^d\rightarrow \R^d$ on $\R^d$ such that 
\[ \forall k,\; \forall y\in B_k \cap B(0,r) , \,\, \phi(y) = y_k\;\text{ and }\;
\forall y\in B(0,r)^c, \,\,  \phi(y) = 0.\]
Then, 
\[\phi \# \mu =  \sum_{k=1}^N \mu(B_k\cap B(0,r)) \delta_{y_k} + \mu(B(0,r)^c)\delta_0 \] and  
\begin{eqnarray*}
W_p^p(\phi \# \mu,\mu) &\leq& \int_{\R^d} \|y - \phi(y)\|^pd\mu(y) \\ &\leq& \epsilon^p\int_{B(0,r)}d\mu(y) + \int_{B(0,r)^c} \|y\|^p d\mu(y) \leq \epsilon^p+ \epsilon^p = 2\epsilon^p,
\end{eqnarray*}
which finishes the proof.
\end{proof}

\subsection*{Identifiability properties of Gaussian mixture models}

\begin{propstar}[Proposition 2]
The set of finite Gaussian mixtures is identifiable, in the sense that two mixtures $\mu_0 =\sum_{k=1}^{K_0} \pi_0^k \mu_0^k$ and $\mu_1= \sum_{k=1}^{K_1} \pi_1^k \mu_1^k$, written such that all $\{\mu_0^k\}_k$ (resp. all $\{\mu_1^j\}_j$) are pairwise distinct, are equal if and only if $K_0=K_1$ and we can reorder the indexes such that for all $k$, $\pi_0^k = \pi_1^k$, $m_0^k = m_1^k$ and $\Sigma_0^k = \Sigma_1^k$. 
  \label{lemma:unicity_gaussian}
\end{propstar}
This result is classical and the proof is also given here in the Appendix for the sake of completeness. 
\begin{proof}
This proof is an adaptation and simplification of the proof of
Proposition 2 in~\cite{yakowitz1968}. First, assume that $d=1$ and
that two Gaussian mixtures are equal:
\begin{equation}
\sum_{k=1}^{K_0} \pi_0^k \mu_0^k = \sum_{j=1}^{K_1} \pi_1^j \mu_1^j.\label{eq:GMM_equal}
\end{equation}
We start by identifying the Dirac masses from both sums, so only
non-degenerate Gaussian components remain. 
Writing $\mu_i^k = \mathcal N (m_i^k, (\sigma_i^k)^2)$,
it follows that 
\[\sum_{k=1}^{K_0} \frac{\pi_0^k}{\sigma_0^k} e^{-\frac{(x-m_0^k)^2}{2(\sigma_0^k)^2}} = \sum_{j=1}^{K_1} \frac{\pi_1^j}{\sigma_1^j} e^{-\frac{(x-m_1^j)^2}{2(\sigma_1^j)^2}},\;\;\forall x \in \mathbb{R}.\]
Now, define $k_0 = \mathrm{argmax}_k \sigma_0^k$ and $j_0 =
\mathrm{argmax}_j \sigma_1^j$. If the maximum is attained for several
values of $k$ (resp. $j$), we keep the one with the largest mean
$m_0^k$ (resp. $m_1^j$). Then, when $x\rightarrow +\infty$, we have
the equivalences
\[\sum_{k=1}^{K_0} \frac{\pi_0^k}{\sigma_0^k}
e^{-\frac{(x-m_0^k)^2}{2(\sigma_0^k)^2}} \underset{x\to+\infty}{\sim} \frac{\pi_0^{k_0}}{\sigma_0^{k_0}} e^{-\frac{(x-m_0^{k_0})^2}{2(\sigma_0^{k_0})^2}} \;\;\text{ and } \sum_{j=1}^{K_1} \frac{\pi_1^j}{\sigma_1^j} e^{-\frac{(x-m_1^j)^2}{2(\sigma_1^j)^2}} \underset{x\to+\infty}{\sim} \frac{\pi_1^{j_0}}{\sigma_1^{j_0}} e^{-\frac{(x-m_1^{j_0})^2}{2(\sigma_1^{j_0})^2}}.\]
Since the two sums are equal, these two terms must also be equivalent when $x\rightarrow +\infty$, which implies  necessarily that $\sigma_0^{k_0} = \sigma_1^{j_0}$, $m_0^{k_0} = m_1^{j_0}$ and $\pi_0^{k_0} = \pi_1^{j_0}$. Now, we can remove these two components from the two sums and we obtain
 \[\sum_{k=1\dots K_0,\; k \neq k_0} \frac{\pi_0^k}{\sigma_0^k} e^{-\frac{(x-m_0^k)^2}{2(\sigma_0^k)^2}} = \sum_{j=1\dots K_1,\; j\neq j_0} \frac{\pi_1^j}{\sigma_1^j} e^{-\frac{(x-m_1^j)^2}{2(\sigma_1^j)^2}},\;\;\;\forall x \in \mathbb{R}.\]
We can start over and show recursively that all components are equal.

For $d>1$, assume once again that two Gaussian mixtures $\mu_0$ and $\mu_1$ are equal, written as in Equation~\eqref{eq:GMM_equal}. 
The projection of this equality  yields 
\begin{equation}
 \sum_{k=1}^{K_0} \pi_0^k \mathcal N (\langle m_0^k,\xi\rangle ,\xi^t\Sigma_0^k\xi) =  \sum_{j=1}^{K_1} \pi_1^j \mathcal N (\langle m_1^j,\xi\rangle ,\xi^t\Sigma_1^j\xi),\;\;\; \forall \xi\in \mathbb{R}^d.\label{eq:GMM_projected} 
\end{equation}
At this point, observe that for some values of $\xi$, some of these projected components may not be pairwise distinct anymore, so we cannot directly apply the result for $d=1$ to such mixtures. However, since the pairs $(m_0^k,\Sigma_0^k)$ (resp. $(m_1^j,\Sigma_1^j)$) are all distinct, then for $i=0,1$, the set
\[\Theta_i = \bigcup_{1\leq k,k'\leq K_i}\left\{\xi \;\;\text{s.t.}\;  \langle m_i^k - m_i^{k'},\xi\rangle  = 0 \;\;\text{and}\;\; \xi^t\left(\Sigma_i^k - \Sigma_i^{k'}\right)\xi=0\right\}\]
is of Lebesgue  measure $0$ in $\mathbb{R}^d$. For any $\xi$ in $\mathbb{R}^d\setminus\Theta_0\cup\Theta_1$, the pairs $\{(\langle m_0^k,\xi\rangle,\xi^t\Sigma_0^k\xi)\}_{k}$ (resp. $\{(\langle m_1^j,\xi\rangle,\xi^t\Sigma_1^j\xi)\}_{j}$) are pairwise distinct. 
Consequently, using the first part of the proof (for $d=1$), we can deduce that $K_0 = K_1$ and that
\begin{equation}
  \label{eq:inclusion}
\mathbb{R}^d\setminus \Theta_0\cup\Theta_1 \subset \bigcap_k\bigcup_j \Xi_{k,j}
\end{equation}
where 
\begin{equation*}
 \Xi_{k,j} = \left\{\xi,\;\;\text{s.t.}\;  \pi_0^k = \pi_1^j,\;\;\langle m_0^k - m_1^{j},\xi\rangle  = 0 \;\;\text{and}\;\; \xi^t\left(\Sigma_0^k - \Sigma_1^{j}\right)\xi=0\right\}.\label{eq:2}
\end{equation*}
Now, assume that the two sets $\{(\pi_0^k,m_0^k,\Sigma_0^k)\}_k$ and $\{(\pi_1^j,m_1^j,\Sigma_1^j)\}_j$ are different. Since each of these sets is composed of different triplets, it is equivalent to assume that there exists $k$ in  $\{1,\dots K_0\}$ such that $(\pi_0^k,m_0^k,\Sigma_0^k)$ is different from all triplets $(\pi_1^j,m_1^j,\Sigma_1^j)$. In this case, the sets $\Xi_{k,j}$ for $j = 1,\dots K_0$ are all of Lebesgue measure $0$ in $\mathbb{R}^d$, which contradicts~\eqref{eq:inclusion}. 
We conclude that the sets $\{(\pi_0^k,m_0^k,\Sigma_0^k)\}_k$ and $\{(\pi_1^j,m_1^j,\Sigma_1^j)\}_j$ are equal. 
\end{proof}

\bibliographystyle{siamplain}
\bibliography{biblio-otgmm.bib}

\end{document}